\documentclass[12pt]{article}
\usepackage[utf8]{inputenc}
\usepackage{amsmath}
\usepackage{amsthm}
\usepackage{amssymb}
\usepackage{mathtools}
\usepackage{stmaryrd}
\usepackage{bbm}

\DeclareMathOperator*{\slim}{s-lim}
\usepackage{geometry}

\numberwithin{equation}{section}

\theoremstyle{definition}
\newtheorem{rem}{Remark}[section]
\theoremstyle{plain}  
\newtheorem{thm}[rem]{Theorem}
\newtheorem{prop}[rem]{Proposition}
\newtheorem{lem}[rem]{Lemma}
\newtheorem{cor}[rem]{Corollary}
\newcommand{\K}{\mathbb{K}}
\newcommand{\R}{\mathbb{R}}
\newcommand{\C}{\mathbb{C}}
\newcommand{\Q}{\mathbb{Q}}
\newcommand{\N}{\mathbb{N}}
\newcommand{\dx}{\, dx}
\newcommand{\dy}{\, dy}

\renewcommand{\phi}{\varphi}

\newcommand{\dom}{\operatorname{dom}}

\newcommand{\Rea}{\operatorname{Re}}

\newcommand{\oneover}[1]{\frac{1}{\raisebox{0.1em}{$\scriptscriptstyle#1$}}}
\newcommand{\pq}{\oneover p-\oneover q}

\newcommand{\ind}{\mathbf{1}}
\newcommand{\DL}{\Delta_{\textnormal D}}
\newcommand{\rma}{{\textnormal{(a)\ }}}
\newcommand{\rmb}{{\textnormal{(b)\ }}}
\newcommand{\rmc}{{\textnormal{(c)\ }}}

\newcommand*\setcol{\nobreak\mskip2mu\mathclose{}
  \mathopen{;}\penalty300\mskip6mu plus3mu minus1mu\relax}

\title{Sharp Gaussian upper bounds for \\ Schrödinger semigroups on the half-line}
\author{Paul Holst and Hendrik Vogt}
\date{}

\begin{document}

\maketitle

\begin{abstract}
In 1998, V. Liskevich and Y. Semenov proved sharp Gaussian upper bounds for Schrödinger semigroups on $\R^3$ with potentials satisfying a global Kato class condition. Using similar basic ideas we show sharp Gaussian upper bounds for Schrödinger semigroups on the half-line, also assuming a suitable global Kato class condition. Our proof strategy includes a new technique of weighted ultracontractivity estimates. 

\vspace{8pt}


\vspace{2pt}

\end{abstract}

\section{Introduction}\label{ch-intro}

In this paper we
show a kernel estimate for the $C_0$-semigroup generated by the Schrödinger operator $\Delta + V$ on $(0,\infty)$ 
for potentials $V$ satisfying the integral condition given by~\eqref{int-cond} below.
This integral condition amounts to a global Kato class condition with respect to the Dirichlet Laplacian on $(0,\infty)$; see Remark~\ref{global-Kato}.

Throughout the paper let $\K \in \{\R,\C\}$ and $d \in \N$.
Let $\Omega \subseteq \R^d$ be open, and let $T$ be the $C_0$-semigroup on $L_2(\Omega)$ generated by the Dirichlet Laplacian on $\Omega$, i.e., by the operator $\DL$ defined by
\begin{align*}
\dom(\DL)&:= \{ u \in H_0^1(\Omega) \setcol \ \Delta u \in L_2(\Omega) \} \\
\DL u &:=\Delta u.
\end{align*}

Our main result is as follows.

\begin{thm}\label{thm-main}
Let $d = 1$ and $\Omega = (0,\infty)$. Let $V \colon (0,\infty) \rightarrow \R$ be measurable, and assume that there exists $\alpha \in (0,1)$ such that
\begin{equation}\label{int-cond}
\int_0^{\infty} x|V(x)| \dx \leqslant \alpha.
\end{equation}
Then $V$ is $T$-admissible, for every $t > 0$ the operator $T_V(t)$ has an integral kernel $k_t^V \in L_{\infty}((0,\infty) \times (0,\infty))$, and there exists $c > 0$ such that
\begin{equation}\label{the-kernel-estimate}
0 \leqslant k_t^V(x,y) \leqslant c \left( 1 \wedge \left( \frac{xy}{t}\left(1 + \frac{(x-y)^2}{4t}\right)^{3\alpha/4} \right) \right) t^{-1/2} e^{- (x-y)^2/4t}
\end{equation}
for all $t > 0$ and a.e.\ $x,y \in (0,\infty)$.
\end{thm}

Recall that $k_t^V$ being an integral kernel of the operator $T_V(t)$ means that $T_V(t)f(x) = \int_0^\infty k_t^V(x,y)f(y)\,dy$ for all $f\in L_2(0,\infty)$ and a.e.\ $x\in(0,\infty)$.
In Theorem~\ref{thm-necessary} we will show that in the case $V\leqslant0$ the integral condition~\eqref{int-cond} for some $\alpha<\infty$ is in fact necessary for the validity of the kernel estimate~\eqref{the-kernel-estimate}.

For $V=0$ the kernel $k_t^0$ satisfies
\[
\frac12 \Bigl(1 \wedge \frac{xy}{t} \Bigr) (4\pi t)^{-1/2}e^{-(x-y)^2/4t} \leqslant k_t^0(x,y) \leqslant \Bigl(1 \wedge \frac{xy}{t} \Bigr) (4\pi t)^{-1/2}e^{-(x-y)^2/4t},
\]
see~\eqref{kernel-est-0}.
The term $1 \wedge \frac{xy}{t}$ describes the boundary behaviour: it is `small' if $x$ and~$y$ are `close' to the boundary $\{0\}$ of~$(0,\infty)$ (where the meaning of `close' depends on~$t$).
Observe that the Gaussian upper bound~\eqref{the-kernel-estimate} is `almost sharp'; in particular, the factor~4 in the exponential term $e^{-(x-y)^2/4t}$ is the same as in $k_t^0$. In comparison with $k_t^0$, the major difference is the polynomial correction factor $\bigl(1 + (x-y)^2/4t\bigr)\rule{0pt}{1.6ex}^{3\alpha/4}$ for the term~$\frac{xy}{t}$.
It is an open question whether this term can be avoided.

In the general context of ultracontractive self-adjoint $C_0$-semigroups on metric measure spaces with the doubling property it is known that a Davies-Gaffney estimate implies the kernel estimate
\[
k_t(x,y) \leqslant c \mu\bigl(B(x,\sqrt t)\bigr)^{-1/2} \mu\bigl(B(y,\sqrt t)\bigr)^{-1/2} \left(1+\frac{\rho(x,y)^2}{4t}\right)^{(D-1)/2} e^{-\rho(x,y)^2/4t},
\]
where $\mu$ is the measure, $\rho$ the metric and $D$ the dimension defined via the polynomial volume growth that holds because of the doubling property; see \cite[Remark~3 after Theorem~5]{sik04}.
(Note that in \cite[p.\,651, line~-2]{sik04} the exponent should read $D-1$ and not $(D-1)/2$.)
In our context, the above result would lead to a kernel estimate similar to~\eqref{the-kernel-estimate}, but with the (larger!)\ exponent $1$ instead of $3\alpha/4$ in the polynomial correction factor since we will be working with the measure $x^2\dx$ on $(0,\infty)$, for which one obtains $D=3$.
(We refer to \cite{mol75} for an example showing that in general the exponent $(D-1)/2$ is sharp.)

We point out that the~1 in the boundary term $1 \wedge \frac{xy}{t}$ goes without a correction factor.
To that effect, the estimate~\eqref{the-kernel-estimate} in Theorem~\ref{thm-main} is in the spirit of \cite[Cor.~1]{LiSe98}, one of the few results on kernel estimates that do not include a polynomial correction factor either. There is vast literature on kernel estimates including a polynomial correction factor; see for instance \cite{DaPa89} and \cite{Ouh06}. In \cite{Ouh06} it was shown that in many circumstances, Gaussian upper bounds ``automatically'' improve to sharp bounds
with a polynomial correction factor.

\smallskip

In the remainder of this introduction we explain the major steps in the proof of
Theorem~\ref{thm-main}.
We will first show a kernel estimate for Schrödinger semigroups on an arbitrary open set $\Omega\subseteq\R^d$ under the assumption that the potential $V$
satisfies the form smallness condition
\begin{equation}\label{form-small0}
\int_{\Omega} |V||u|^2 \dx \leqslant \alpha \langle -\DL u,u \rangle = \alpha \|\nabla u\|_2^2
\qquad (u \in \dom(\DL))
\end{equation}
for some $\alpha \in (0,1)$
and that
the
Schrödinger semigroup~$T_V$
satisfies a suitable exponentially weighted
$L_1$-estimate; see Theorem~\ref{kernel-bound-exponential}.
More precisely, we show the kernel estimate
\begin{equation}\label{kernel-estimate-exponential}
0 \leqslant k_t^V(x,y) \leqslant c t^{-d/2} e^{-|x-y|^2/4t} \qquad (t>0,\ \text{a.e. } x,y \in \Omega).
\end{equation}
Note that there is no restriction on the dimension~$d$ in this kernel estimate. It will be proved by means of weighted ultracontractivity estimates
and the well-known Davies trick (see \cite{Dav87});
in Subsection~\ref{sec-ultracon-bounds} we develop new techniques that allow us to avoid the use of (weighted) sesquilinear forms that is common in this context.

In Section~\ref{ch-kernel-bounds-phragmen} we consider the special case in which $\Omega$ is the positive half-space, i.e.,
\[
\Omega = \Omega_0:= \begin{cases} 
(0,\infty) \qquad  \qquad &\text{if } d=1, \\
(0,\infty) \times \R^{d-1} &\text{if } d \geqslant 2.
\end{cases}
\]
In Subsection~\ref{ch-ultra-stoch-heat-semi} we prove exponentially weighted $L_1$- and ultracontractivity estimates for the heat semigroup~$T$ with suitable weights.
They will be needed later in that section, where we show a second kernel estimate for the Schrödinger semigroup~$T_V$,
\begin{equation}\label{kernel-estimate-double}
0 \leqslant k_t^V(x,y) \leqslant cx_1y_1t^{-(d/2 + 1)} \left(1 + \frac{|x-y|^2}{4t} \right)^{\alpha(d+2)/4} e^{-|x-y|^2/4t}
\end{equation}
for all $t>0$ and a.e.\ $x,y \in \Omega_0$,
again using the technique of weighted ultracontractivity estimates and Davies trick.
This kernel estimate will be proved under condition~\eqref{form-small0} and the assumption that $T_V$ satisfies suitable weighted $L_1$-estimates; see Theorem~\ref{ultra-Schr-semi-double}.

Taking the minimum of \eqref{kernel-estimate-exponential} and~\eqref{kernel-estimate-double} yields the kernel estimate
\begin{equation}\label{kernel-estimate-exponential-double}
0 \leqslant k_t^V(x,y) \leqslant c \left(1 \wedge \left( \frac{x_1y_1}{t} \left(1 + \frac{|x-y|^2}{4t} \right)^{\alpha(d+2)/4} \right) \right) t^{-d/2}e^{- |x -y|^2/4t} 
\end{equation}
for all $t>0$ and a.e.\ $x,y \in \Omega_0$, see Corollary~\ref{kernel-bound-Schr-semi}. In the case $d = 1$, \eqref{kernel-estimate-exponential-double} is exactly the kernel estimate stated in Theorem~\ref{thm-main} above. Therefore, for the proof of Theorem~\ref{thm-main} it remains to show that \eqref{int-cond} implies the assumptions indicated above for the validity of \eqref{kernel-estimate-exponential} and~\eqref{kernel-estimate-double};
this will be done in Section~\ref{ch-estimates-exponential-double}.

The basic idea for the proof of~\eqref{kernel-estimate-exponential} is similar as in \cite{LiSe98}.
There the case of Schrödinger semigroups on~$\Omega = \R^3$ is treated, based on the following main observation:
for every $\xi \in \R^3$ the Green kernel $G_\xi$ of the resolvent $(|\xi|^2 - \DL)^{-1}$ of the Dirichlet Laplacian satisfies the weighted estimate $e^{\xi \cdot x}G_{\xi}(x,y)e^{- \xi \cdot y} \leqslant \lim_{\xi\to0} G_\xi(x,y)$, for all $x,y\in\R^3$.
In our case $\Omega=(0,\infty)\subseteq\R^1$ an analogous weighted estimate is valid, as we will show in Subsection~\ref{sec-kernel-resolvent}.

\medskip

\emph{Notation.}
For
$x,y \in [-\infty,\infty]$
we denote by $x \wedge y$ and $x \vee y$ the minimum and the maximum of $x$ and $y$, respectively. Similarly, if $A$ is a set and $f,g \colon A \rightarrow [-\infty,\infty]$, then
$f \wedge g$ and $f \vee g$ is the pointwise minimum and maximum of $f$ and $g$, respectively. Moreover we write $[f \geqslant g]:= \{x \in A \setcol f(x) \geqslant g(x)\}$; the sets $[f \leqslant g]$, $[f> g]$ and $[f = g]$ are defined in a similar way. 

We write $\C_+ := \{z \in \C \setcol \Rea z > 0\}$ for the right complex half-plane.  

If $X,Y$ are two Banach spaces, then we denote by $\mathcal{L}(X,Y)$ the space of all bounded linear operators from $X$ to $Y$. In the case $X=Y$ we simply write $\mathcal{L}(X):= \mathcal{L}(X,X)$. A sequence $(A_n)_{n \in \N}$ in $\mathcal{L}(X,Y)$ is called \textbf{strongly convergent} to $A \in \mathcal{L}(X,Y)$, abbreviated $A = \slim_{n \to \infty} A_n$, if $Ax = \lim_{n \to \infty} A_nx$ for
 all $x \in X$.  

Now let $(\Omega,\mathcal{A},\mu)$ be a measure space. We write $\ind_A$ for the indicator function of a set $A \in \mathcal{A}$. For a measurable function $\rho\colon\Omega \rightarrow [0,\infty)$ we denote by $\rho \mu$ the measure that has density $\rho$ with respect to $\mu$.  If $p \in [1,\infty)$ and $V:\Omega \rightarrow \K$ is measurable, then we denote by $V$ the associated multiplication operator on $L_p(\mu)$ as well. Furthermore, if $p,q,r \in [1,\infty]$ and $A\colon L_p(\mu) \rightarrow L_p(\mu)$ is linear, then we write 
\[
   \|A\|_{q \rightarrow r} := \|A\|_{L_q(\mu) \rightarrow L_r(\mu)} := \sup\bigl\{\|Af\|_r \setcol f \in L_p \cap L_q(\mu),\ \|f\|_q \leqslant 1\bigr\}
   \ (\in [0,\infty]).
\]
More generally, if $\rho_1,\rho_2\colon\Omega \rightarrow (0,\infty)$ are measurable, we also write
\begin{align*}
&\|\rho_1^{-1} A \rho_1\|_{L_q(\rho_2\mu) \rightarrow L_r(\rho_2\mu)} \\
& \quad 
:= \sup\bigl\{\|\rho_1^{-1} A \rho_1 f\|_{L_r(\rho_2\mu)} \setcol f \in L_q(\rho_2\mu),\ \|f\|_{L_q(\rho_2\mu)} \leqslant 1,\ \rho_1 f \in L_p(\mu)\bigr\}.
\end{align*}
In the case $\rho_2 = 1$ we simply write $\|\rho_1^{-1} A \rho_1\|_{q \rightarrow r} := \|\rho_1^{-1} A \rho_1\|_{L_q(\mu) \rightarrow L_r(\mu)}$.

Finally, let $H$ be a Hilbert space. Then we denote the scalar product of $x,y$ in~$H$ by $\langle x,y \rangle_H$. We also write $\langle x,y \rangle$ if it is obvious from the context that this is the scalar product on~$H$.

\section{Kernel estimates for Schrödinger semigroups via weighted ultracontractivity estimates}\label{ch-kernel-bounds-exponential}

The ultimate goal of this section is the proof of the kernel estimate~\eqref{kernel-estimate-exponential}. To achieve this goal,
we will show that the 
Schrödinger semigroup~$T_V$ provided with exponential weights is ultracontractive and then use the well-known Davies trick.
The main tool to prove this ultracontractivity will be given in Subsection~\ref{sec-ultracon-bounds},
where we show more generally for a positive self-adjoint $C_0$-semigroup $T$ on an arbitrary measure space and a suitable $T$-admissible potential~$V$  that a weighted ultracontractivity estimate for~$T$  implies a similar estimate for the perturbed $C_0$-semigroup~$T_V$; see Theorem~\ref{weighted-ultracon}.
The kernel estimate~\eqref{kernel-estimate-exponential} will then be proved in Subsection~\ref{sec-kernel-bounds-exponential}. In Subsection~\ref{sec-admis} we first recall some important
properties of admissible potentials. Subsection~\ref{sec_lc} provides an interpolation inequality for positive $C_0$-semigroups, which we will prove using logarithmically convex funtions. It will be important for the proof of the perturbation results in Subsection \ref{sec-ultracon-bounds}.

\subsection{Admissibility of real-valued measurable potentials}\label{sec-admis}

The purpose of this subsection is to recall the notion of admissibility for real-valued measurable potentials, which was introduced in \cite[Sec.~2]{Vo86}. For this let $(\Omega,\mu)$ be a measure space, $p \in [1,\infty)$ and $T$ a positive $C_0$-semigroup on $L_p(\mu)$ with generator $A$. 

For $V \in L_{\infty}(\mu)$, the $C_0$-semigroup generated by $A-V$ is denoted by $T_V$.
If $V,W \in L_{\infty}(\mu)$,\, $V \geqslant W$, then $0 \leqslant T_V(t) \leqslant T_W(t)$ for all $t \geqslant 0$ (see \cite[Rem.~2.1(a)]{Vo86}); in particular, $T_V$ is positive.

If now $V\colon\Omega \rightarrow [0,\infty)$ is measurable, then we have
\[
0 \leqslant T_{V \wedge (n+1)}(t) \leqslant T_{V \wedge n}(t) \qquad (t \geqslant 0)
\]
and, therefore, $T_V(t):= \slim_{n \rightarrow \infty} T_{V \wedge n}(t)$ exists for all $t \geqslant 0$ by dominated convergence. The function $V$ is called \textit{$T$-admissible} if $T_V\colon[0,\infty) \rightarrow \mathcal{L}(L_p(\mu))$ thus defined is a $C_0$-semigroup. 

Similarly, a measurable function $V\colon\Omega \rightarrow (-\infty,0]$ is called \textit{$T$-admissible} if $T_V(t):= \slim_{n \rightarrow \infty}T_{V \vee (-n)}(t)$ exists for all $t \geqslant 0$ and $T_V\colon[0,\infty) \rightarrow \mathcal{L}(L_p(\mu))$ is a $C_0$-semigroup. 

More generally, a measurable function $V\colon\Omega \rightarrow \R$ is called \textit{$T$-admissible} if $V^+$ and $-V^-$ are $T$-admissible. In this case $T_V(t):= \slim_{n \rightarrow \infty} T_{(V \wedge n) \vee (-n)}(t)$ exists and $T_V\colon[0,\infty) \rightarrow \mathcal{L}(L_p(\mu))$ thus defined is a $C_0$-semigroup; see \cite[Thm.~2.6]{Vo88}. (Note that for $V \in L_{\infty}(\mu)$ (resp. $V \geqslant 0$, $V \leqslant 0$) the two definitions of $T_V$ given above coincide.)

The following basic properties of admissible potentials will be used througout without further notice.
Let $V\colon\Omega \rightarrow \R$ be $T$-admissible. Then $T_V$ is positive because $T_{(V \wedge n) \vee (-n)}$ is positive for all $n \in \N$.
If $W \colon \Omega \to \R$ is $T$-admissible and $W \leqslant V$, then $T_V(t) \leqslant T_W(t)$ for all $t \geqslant 0$; see \cite[Remark~2.7]{Vo88}. Finally, if $p=2$ and $T$ is self-adjoint, then $T_V$ is self-adjoint. Indeed, $A - (V \wedge n) \vee (-n)$ is self-adjoint for all $n \in \N$, and thus $T_V(t) = \slim_{n \rightarrow \infty} T_{V \wedge n}(t)$ is self-adjoint for all $t \geqslant 0$.

Note that for $p \in [1, \infty)$ and a measurable function $m \colon \Omega \to (0,\infty)$, the mapping $L_p(m^p\mu) \ni f \mapsto m f \in L_p(\mu)$ is an isometric lattice isomorphism.
Thus, if $T$ is a positive $C_0$-semi\-group on $L_p(\mu)$, then we can define a positive $C_0$-semi\-group $T^m$ on $L_p(m^p\mu)$ by
\[
T^m(t)f:= m^{-1}T(t)m f \qquad (t \geqslant 0, \ f \in L_p(m^p\mu)).
\] 
In the next lemma we characterize $T^m$-admissibility of a potential $V$ by means of the $C_0$-semi\-group~$T$.

\begin{lem}\label{admis-char}
Let $p \in [1,\infty)$, and let $T$ be a positive $C_0$-semigroup on $L_p(\mu)$ with generator~$A$. Let $m\colon \Omega \to (0,\infty)$ and $V: \Omega \to \R$ be measurable. Then $V$ is $T^m$-admissible if and only if $V$ is $T$-admissible. Moreover, in this case one has $(T^m)_V = (T_V)^m$.
\end{lem}
\begin{proof}
Observe that the generator $A^m$ of~$T^m$ is given by
\begin{align*}
\dom(A^m)&:= \{f \in L_p(m^{p}\mu) \setcol m f \in \dom(A) \}, \\
A^mf &:= m^{-1} Am f \qquad (f \in \dom(A^m)).
\end{align*}
For $W \in L_{\infty}(\mu) = L_{\infty}(m^p\mu)$ it is straightforward to show that $A^m + W = (A + W)^m$ and hence $(T^m)_W = (T_W)^m$.
From this identity and the definition of admissibility one easily infers that $\pm V^{\pm}$ is $T$-admissible if and only if $\pm V^{\pm}$ is $T^m$-admissible, using that $L_p(m^p\mu) \ni f \mapsto m f \in L_p(\mu)$ is an isometry. This proves the first assertion.
The second assertion then follows from the identity $(T^m)_{(V \wedge n) \vee (-n)} = (T_{(V \wedge n) \vee (-n)})^m$.
\end{proof}

To conclude this subsection, we show that in the case $p=2$, a potential~$V$ is $T$-admissible if it is form small with respect to the generator of~$T$.

\begin{prop}\label{admissibility}
Let $T$ be a positive $C_0$-semigroup on $L_2(\mu)$ and $A$ its generator. Let $V\colon \Omega \rightarrow \R$ be measurable, and assume that 
\begin{equation}\label{form-small}
\int_{\Omega} |V||u|^2 \,d\mu \leqslant \Rea \langle -Au,u \rangle \qquad (u \in \dom(A)).  
\end{equation}
Then $V$ is $T$-admissible. Moreover, $T_V$ is contractive.
\end{prop}
\begin{proof}
For every $n \in \N$, \eqref{form-small} implies that
\[
\Rea \bigl\langle \bigl(A + (V^\pm \wedge n)\bigr)u,u \bigr\rangle \leqslant \Rea \langle Au,u \rangle + \int_{\Omega} |V| |u|^2 \,d\mu \leqslant 0 \qquad (u \in \dom(A)),
\]
so $A + (V^\pm \wedge n)$ is dissipative. It follows that $T_{- (V^\pm \wedge n)}$ is contractive for all $n \in \N$,
so \cite[Prop.~2.2]{Vo88} shows that $-V^\pm$ is $T$-admissible.
By \cite[Prop.~3.3(b)]{Vo88} it follows that $V^+$ and hence also $V$ is $T$-admissible.
The contractivity of $T_V$ follows from $T_V \leqslant T_{-V^-}$ since $T_{- (V^- \wedge n)}$ is contractive for all $n \in \N$.
\end{proof}

\subsection{Logarithmically convex functions}\label{sec_lc}

Throughout let $(\Omega,\mu)$ be a measure space.
The aim of this subsection is to prove the following interpolation inequality.

\begin{thm}\label{thm-interpol}
Let $p\in(1,\infty)$, let $T$ be a positive $C_0$-semigroup on $L_p(\mu)$, and let $V\colon\Omega\to\R$ be $T$-admissible.
Let $t>0$, $p_0,p_1,q_0,q_1\in[1,\infty]$, and assume that $\|T_{jV}(t)\|_{p_j\to q_j} < \infty$ for $j=0,1$.
Let $\theta\in(0,1)$, and define $p_\theta,q_\theta \in [1,\infty]$ by
\begin{equation}\label{ptheta-def}
  \frac{1}{p_\theta} = \frac{1-\theta}{p_0} + \frac{\theta}{p_1}\,, \qquad
  \frac{1}{q_\theta} = \frac{1-\theta}{q_0} + \frac{\theta}{q_1}\,.
\end{equation}
Then
\[
  \|T_{\theta V}(t)\|_{p_\theta\to q_\theta} \leqslant \|T(t)\|_{p_0\to q_0}^{1-\theta} \|T_V(t)\|_{p_1\to q_1}^\theta\,.
\]
\end{thm}

For bounded potentials~$V$ this inequality is well-known and can be proved by means of the Stein interpolation theorem.
The problem is that the inequality does not easily carry over from bounded to admissible potentials since one only has strong convergence $T_{V_n}(t) \to T_V(t)$, which does not imply convergence of the operator norms.
To remedy this problem we introduce the notion of logarithmic convexity  that is motivated by \cite{haa07}.
Our presentation largely follows \cite[Sec.~2.3]{Vog10}.

Let $X$ be an ordered vector space, i.e., $X$ is a (real or complex) vector space endowed with a proper convex cone $X_+$ of positive elements,
where $X_+$ being proper means that $X_+\cap(-X_+) = \{0\}$.
Let $I\subseteq\R$ be an interval.
We say that a function $f \colon I\to X$ is \emph{logarithmically convex} if
\begin{equation}\label{lc-def}
  f((1 - \theta)t_0 + \theta t_1)  \leqslant (1-\theta) r^{-\theta} f(t_0) + \theta r^{1-\theta} f(t_1) \qquad (r>0)
\end{equation}
for all $t_0,t_1\in I$ and all $\theta\in(0,1)$.
By choosing $t_0=t_1$ and $r\ne1$ (so that $(1-\theta) r^{-\theta} + \theta r^{1-\theta} > 1$)
we see that a logarithmically convex function $f$ takes its values in $X_+$.

The next lemma implies in particular that $f \colon I\to\R$ is logarithmically convex if and only if
$f\geqslant 0$ and $\ln\circ f$ is convex, where we use the convention $\ln 0 := -\infty$.
Thus, the assertion of Theorem~\ref{thm-interpol} is that $[0,1] \ni \theta \mapsto \|T_{\theta V}(t)\|_{p_\theta\to q_\theta} \in \R$
is logarithmically convex.

\begin{lem}\label{Mmu}
Let $M(\mu)$ be the ordered vector space of all scalar-valued measurable
functions on $\Omega$, where functions are identified if they coincide a.e.,
and $M(\mu)_+ = \{f\in M(\mu) \setcol f\geqslant 0\ \text{a.e.}\}$.

\rma A function $f \colon I\to M(\mu)_+$ is logarithmically convex if and only if
\[
  f((1 - \theta)t_0 + \theta t_1) \leqslant f(t_0)^{1-\theta} f(t_1)^\theta
\]
a.e.\ for all\/ $t_0,t_1\in I$ and all\/ $\theta\in(0,1)$.

\rmb If $f,g \colon I\to M(\mu)_+$ are logarithmically convex, then $t \mapsto f(t)g(t)$ is
logarithmically convex as well.
\end{lem}
\begin{proof}
(a) follows from Young's inequality: For $a,b\geqslant 0$ and $\theta\in(0,1)$ we have
\[
  a^{1-\theta} b^\theta = \bigl(r^{-\theta}a\bigr)^{1-\theta} \bigl(r^{1-\theta}b\bigr)^\theta \leqslant (1-\theta) r^{-\theta} a + \theta r^{1-\theta} b
  \qquad (r>0)
\]
and
\[
  a^{1-\theta} b^\theta = \inf \bigl\{(1-\theta) r^{-\theta} a + \theta r^{1-\theta} b \setcol 0<r\in\Q\bigr\}.
\]

(b) is an immediate consequence of part (a).
\end{proof}

The next result, though being elementary, is the basis of the proof of Theorem~\ref{thm-interpol}.
We assume that $L_p(\mu)$ and $\mathcal{L}(L_p(\mu))$ are endowed with their natural orderings;
in particular, $\mathcal{L}(L_p(\mu))_+$ consists of the positive operators on $L_p(\mu)$. 
Let $\langle \cdot,\cdot \rangle_{p,p'}$ denote the natural bilinear map associated with the the dual pairing $\langle L_p(\mu),L_{p'}(\mu) \rangle$.

\begin{prop}\label{lc-prop}
Let $p\in[1,\infty)$, $f \colon I\to L_p(\mu)$ and $S \colon I\to\mathcal{L}(L_p(\mu))$. Then

\rma $f$ is logarithmically convex if and only if $t \mapsto \langle f(t),g\rangle_{p,p'}$ is logarithmically convex for all $g\in L_{p'}(\mu)_+$;

\rmb $S$ is logarithmically convex if and only if $t \mapsto S(t)h$ is logarithmically convex for all $h\in L_p(\mu)_+$.
\end{prop}
\begin{proof}
(a) is immediate from the following fact: A function $h\in L_p(\mu)$ is in $L_p(\mu)_+$
if and only if $\langle h,g\rangle_{p,p'} \geqslant 0$ for all $g\in L_{p'}(\mu)_+$.

(b) is clear.
\end{proof}

Clearly, if $f \colon I\to X$ is logarithmically convex and $B \colon X\to X$ is
a positive operator, then $B\circ f$ is logarithmically convex.
In the case $X=L_p(\mu)$ we can prove the following more general result.

\begin{lem}\label{composition}
Let $p\in(1,\infty)$, and let $f \colon I\to L_p(\mu)$, $g \colon I\to L_{p'}(\mu)$ and $S \colon I\to\mathcal{L}(L_p(\mu))$
be logarithmically convex. Then

\rma $I \ni t \mapsto \langle f(t),g(t)\rangle_{p,p'} \in \R$ is logarithmically convex.

\rmb $I \ni t \mapsto S(t)f(t) \in L_p(\mu)$ is logarithmically convex.

\rmc $I \ni t \mapsto \langle S(t)f(t),g(t)\rangle_{p,p'} \in \R$ is logarithmically convex.
\end{lem}
\begin{proof}
(a) follows from Lemma~\ref{Mmu}(b) and the positivity of the linear operator $L_1(\mu)\ni f\mapsto \int f\,d\mu \in \R$.

(b) It follows from Proposition~\ref{lc-prop} that $I\ni t\mapsto S(t)' \in \mathcal{L}(L_{p'}(\mu))$ is logarithmically convex.
Therefore, by part~(a), $I\ni t\mapsto \langle S(t)f(t),h\rangle_{p,p'} = \langle f(t),S(t)'h\rangle_{p,p'}  \in \R$ is logarithmically convex,
for all $h\in L_{p'}(\mu)_+$, and the assertion follows by Proposition~\ref{lc-prop}(a).

(c) follows immediately from parts (a) and~(b).
\end{proof}

Now we are ready to prove Theorem~\ref{thm-interpol}.

\begin{proof}[Proof of Theorem~\ref{thm-interpol}]
Fix $t>0$. For $n\in\N$ let $V_n := (V\wedge n)\vee(-n)$.
It follows from \cite[top of p.~121]{Vo88} and Proposition~\ref{lc-prop} that
$[0,1]\ni\theta \mapsto T_{\theta V_n}(t) \in \mathcal{L}(L_p(\mu))$ is logarithmically convex.
Therefore, $[0,1]\ni\theta \mapsto T_{\theta V}(t) = \slim_{n\to\infty} T_{\theta V_n}(t) =: S(\theta) \in \mathcal{L}(L_p(\mu))$
is logarithmically convex.
(Note that $\theta V$ is $T$-admissible for all $\theta\in[0,1]$, by~\cite[Prop.~2.3]{Vo88}.)

To prove the theorem we now show that $[0,1]\ni\theta \mapsto \|S(\theta)\|_{p_\theta\to q_\theta} \in \R$ is logarithmically convex
for any logarithmically convex function $S\colon[0,1]\to \mathcal{L}(L_p(\mu))$ with $\|S(j)\|_{p_j \rightarrow q_j} < \infty$ for $j= 0,1$.
Let $\phi,\psi\in S(\mu)_+$, where $S(\mu)$ denotes the vector space of simple functions on $(\Omega,\mu)$.
Define $f \colon[0,1]\to L_p(\mu)$ by $f(\theta) := \phi^{1/p_\theta} \ind_{[\phi\ne0]}$
and $g \colon[0,1]\to L_{p'}(\mu)$ by $g(\theta) := \psi^{1/q_\theta'} \ind_{[\psi\ne0]}$.
(The indicator functions are only needed for the cases $p_\theta=\infty$, $q_\theta'=\infty$, respectively.)
Then $f,g$ are logarithmically convex, so Lemma~\ref{composition}(c) implies that
\[
  [0,1] \ni \theta \mapsto \langle S(\theta)f(\theta), g(\theta)\rangle_{p,p'} \in \R
\]
is logarithmically convex. It easily follows that
\[
  [0,1] \ni \theta \mapsto \sup_{\phi,\psi\in M}
  \langle S(\theta) \phi^{1/p_\theta} \ind_{[\phi\ne0]}, \psi^{1/q_\theta'} \ind_{[\psi\ne0]} \rangle_{p,p'} \in \R
\]
is logarithmically convex as well, where $M := \{\phi\in S(\mu)_+ \setcol \|\phi\|_1\leqslant 1 \}$.

To complete the proof, it remains to observe that for $q,r\in[1,\infty]$
and a positive operator $B\in\mathcal{L}(L_p(\mu))$ one has
\[
  \|B\|_{q\to r} = \sup_{\phi,\psi\in M}
  \langle B \phi^{1/q} \ind_{[\phi\ne0]}, \psi^{1/r'} \ind_{[\psi\ne0]} \rangle_{p,p'}.
\]
To see this identity, note that for $q<\infty$ the set
$\{\phi^{1/q} \ind_{[\phi\ne0]} \setcol \phi\in M\} = \{h\in S(\mu)_+ \setcol \|h\|_q\leqslant 1 \}$ is dense in $\{h\in L_p(\mu)_+ \setcol \|h\|_q\leqslant 1 \}$.
If $q=\infty$ then $\{\phi^{1/\infty} \ind_{[\phi\ne0]} \setcol \phi\in M\}
= \{\ind_A \setcol A\ \text{measurable},\ \mu(A)<\infty \}$,
and for all $h\in L_p(\mu)_+$ with $\|h\|_\infty\leqslant 1$ there exists an increasing sequence $(A_n)$ of measurable sets of finite measure such that $h\leqslant\lim_{n\to\infty}\ind_{A_n}$.
Moreover
\[
  \|Bf\|_r = \sup_{\psi\in M} \langle Bf, \psi^{1/r'} \ind_{[\psi\ne0]} \rangle_{p,p'}
  \qquad (f\in L_p(\mu)_+),
\]
where in the case $r=\infty$ we use the fact that the set $[Bf\ne0]$ is $\sigma$-finite; note that for $r=1$ the above equality reads $\|Bf\|_1 = \sup\{ \int_A Bf\,d\mu \setcol A\ \text{measurable},\ \mu(A)<\infty \}$.
\end{proof}

\subsection{Weighted ultracontractivity estimates for perturbed $C_0$-semi\-groups on $L_2$}\label{sec-ultracon-bounds}

In this subsection let $\K = \C$, and let $(\Omega,\mu)$ be a measure space. The main result is the next theorem; we state and prove it in greater generality than actually needed for showing Theorem~\ref{thm-main}.

Note that, for a positive $C_0$-semigroup $T$ on $L_2(\mu)$ and measurable $V\colon\Omega \to \R$, the potential $V$ is $T$-admissible if $pV$ is $T$-admissible for some $p>1$, by~\cite[Prop.~2.3]{Vo88}.

\begin{thm}\label{weighted-ultracon}
Let $\rho\colon\Omega\to(0,\infty)$ be measurable, and let $T$ be a positive self-adjoint $C_0$-semigroup on $L_2(\mu)$ satisfying
\[
  \|\rho^{-\alpha}T(t)\rho^\alpha\|_{1\to\infty} \leqslant ct^{-\nu}e^{\alpha^2t}, \quad \|\rho^{-\alpha}T(t)\rho^\alpha\|_{2 \to 2} \leqslant e^{\alpha^2t}
  \qquad (t>0,\ \alpha\in\R)
\]
for some $c,\nu>0$. Let $V\colon \Omega \to \R$ be measurable, and assume that
there exist $p>1$, $M\geqslant 1$ and $r\geqslant1$ such that $pV$ is $T$-admissible,
$\|T_{pV}(t)\|_{2\to2} \leqslant 1$,
\[
  \|\rho^{-\alpha}T_V(t)\rho^\alpha\|_{1\to1} \leqslant Me^{\alpha^2t}, \quad
  \|\rho^{-\alpha}T_V(t)\rho^\alpha\|_{\infty\to\infty} \leqslant Me^{r\alpha^2t}
  \qquad (t>0,\ \alpha>0).
\]
Then there exists $\tilde c>0$ such that
\[
  \|\rho^{-\alpha}T_V(t)\rho^\alpha\|_{1\to\infty}
  \leqslant \tilde ct^{-\nu} \bigl(1+(r-1)\alpha^2t\bigr)^{\nu/2p} e^{\alpha^2t}
  \qquad (t>0,\ \alpha\in\R).
\]
\end{thm}

We first show an extrapolation result that extends \cite[Lemme~1]{cou90}, for semigroups that are not necessarily self-adjoint; in the proof of Theorem~\ref{weighted-ultracon} it will be applied to the semigroup $t \mapsto e^{-\alpha^2t}\rho^{-\alpha}T_V (t)\rho^{\alpha}$.

\begin{prop}\label{coulhon-extrapol}
Let $T$ be a one-parameter semigroup on $L_2(\mu)$.
Assume that there exist $1\leqslant p<q\leqslant \infty$, $c,\nu>0$, $M\geqslant 1$ and $\omega\geqslant 0$ such that
\[
\|T(t)\|_{1\to1} \leqslant M, \quad
\|T(t)\|_{\infty\to\infty} \leqslant Me^{\omega t}, \quad
\|T(t)\|_{p\to q} \leqslant ct^{-\nu(\pq)} \qquad (t>0).
\]
Then $T$ is ultracontractive,
\[
  \|T(t)\|_{1\to\infty} \leqslant \tilde ct^{-\nu} (1+\omega t)^{\nu/q} \qquad (t>0),
\]
with a constant $\tilde c>0$ depending only on $p,q,c,\nu$ and~$M$.
\end{prop}
\begin{proof}
(i) In the first step we show that
\begin{equation}\label{1q-est}
  \|T(t)\|_{1\to q} \leqslant c_0 t^{-\nu/q'} \qquad (t>0),
\end{equation}
with $c_0 = 2^{\alpha\nu/q'} c^\alpha M$, $\alpha = \frac{1}{q'}\frac{1}{\pq}$.
(Actually this is the first step in the proof of the extrapolation result \cite[Lemme~1]{cou90},
except that there no explicit constant $c_0$ is given.
Since it will be crucial for us that $c_0$ depends only on $p,q,c,\nu$ and~$M$,
we present the extrapolation argument in full detail.
Moreover, unlike \cite{cou90} we do not assume the measure~$\mu$ to be $\sigma$-finite.)

Fix $f\in L_1\cap L_\infty(\mu)$, $\|f\|_1 \leqslant 1$ and $t_0>0$.
Since by assumption
\[
  \|T(t)f\|_q \leqslant ct^{-\nu(\pq)}\|f\|_p \leqslant ct_0^{\nu/p'}\|f\|_p t^{-\nu/q'} \qquad (0<t\leqslant t_0),
\]
there exists $\tilde c_0\geqslant0$ (depending on $f$ and $t_0$!)\ such that
\[
  \|T(t)f\|_q \leqslant \tilde c_0 t^{-\nu/q'} \qquad (0<t\leqslant t_0);
\]
we choose the minimal constant $\tilde c_0$ making this estimate valid.
Let $\theta\in[0,1]$ be such that $\frac{1-\theta}{1} + \frac{\theta}{q} = \frac{1}{p}$,
i.e. $\frac{1}{p'} = \frac{\theta}{q'}$. Then by H\"older's inequality
\[
  \|T(t/2)f\|_p \leqslant \|T(t/2)f\|_1^{1-\theta} \|T(t/2)f\|_q^\theta
                \leqslant M^{1-\theta} \tilde c_0^{\mkern1mu\theta}(t/2)^{-\theta\nu/q'}
\]
and thus
\begin{align*}
\|T(t)f\|_q
 &\leqslant \|T(t/2)\|_{p\to q} \|T(t/2)f\|_p \\
 &\leqslant c(t/2)^{-\nu(\pq)} \cdot M^{1-\theta} \tilde c_0^{\mkern1mu\theta}(t/2)^{-\nu/p'}
  = c M^{1-\theta} \tilde c_0^{\mkern1mu\theta}\, 2^{\nu/q'} t^{-\nu/q'}
\end{align*}
for all $0<t\leqslant t_0$.
By the choice of $\tilde c_0$ it follows that $\tilde c_0 \leqslant c M^{1-\theta} \tilde c_0^{\mkern1mu\theta}\, 2^{\nu/q'}$
and hence $\tilde c_0^{1-\theta} \leqslant c M^{1-\theta}2^{\nu/q'}$.
Thus \eqref{1q-est} is valid with $c_0 = 2^{\alpha\nu/q'} c^\alpha M$,
$\alpha = \frac{1}{1-\theta} = \frac{1}{q'}\frac{1}{\pq}$ as asserted.

(ii) The one-parameter semigroup $S$ on $L_2(\mu)$ defined by $S(t) := e^{-\omega t}T(t)^*$
satisfies $\|S(t)\|_{1\to1} \leqslant M$ and
$\|S(t)\|_{q'\to p'} \leqslant ct^{-\nu(\pq)}$ for all $t>0$.
Applying step~(i) to $S$ (and noting $\frac1p-\frac1q = \frac{1}{q'}-\frac{1}{p'}$) we obtain
\[
  \|T(t)\|_{p\to\infty} = \|T(t)^*\|_{1\to p'} \leqslant c_1 t^{-\nu/p} e^{\omega t} \qquad (t>0),
\]
with $c_1 = 2^{\beta\nu/p} c^\beta M$, $\beta = \frac{1}{p}\frac{1}{\pq}$.
By the Riesz-Thorin interpolation theorem we infer that
\begin{equation}\label{qinfty-est}
  \|T(t)\|_{q\to\infty}
  \leqslant \|T(t)\|_{p\to\infty}^{p/q} \|T(t)\|_{\infty\to\infty}^{1-p/q}
  \leqslant c_2 t^{-\nu/q} e^{\omega t} \qquad (t>0),
\end{equation}
with $c_2 = c_1^{p/q}M^{1-p/q}$.
Combining \eqref{1q-est} and~\eqref{qinfty-est} we conclude that
\[
\|T(t)\|_{1\to\infty} \leqslant \|T((1-\varepsilon)t)\|_{1\to q} \|T(\varepsilon t)\|_{q\to\infty}
\leqslant c_0 ((1-\varepsilon)t)^{-\nu/q'} \cdot c_2 (\varepsilon t)^{-\nu/q} e^{\omega\varepsilon t}
\]
for all $t>0$, $\varepsilon\in(0,1)$.  For $\varepsilon:= 1/(2+\omega t) \; (\leqslant 1/2)$ we obtain
\[
  \|T(t)\|_{1\to\infty}
  \leqslant c_0 2^{\nu/q'} t^{-\nu} \cdot c_2 (2+\omega t)^{\nu/q} e \qquad (t>0),
\]
and the asserted ultracontractivity estimate follows, with $\tilde c = 2^\nu e c_0 c_2$.
\end{proof}

In the proof of Theorem~\ref{weighted-ultracon} we will use the following result
to show an $L_2$-bound for the operators $\rho^{-\alpha}T_{pV}(t)\rho^\alpha$.

\begin{prop}\label{weighted-ultracon2}
Let $\rho\colon\Omega\to(0,\infty)$ be measurable, and let $T$ be a positive self-adjoint $C_0$-semigroup on $L_2(\mu)$ satisfying
\[
  \|\rho^{-\alpha}T(t)\rho^\alpha\|_{2\to2} \leqslant Me^{\alpha^2t+\omega t}
  \qquad (t\geqslant 0,\ \alpha\in\R)
\]
for some $M \geqslant 1$, $\omega\in\R$. Let $V \colon \Omega \to \R$ be $T$-admissible, and assume that $T_V$ is contractive.
Then
\[
  \|\rho^{-\alpha}T_V(t)\rho^\alpha\|_{2\to2} \leqslant e^{\alpha^2t}
  \qquad (t\geqslant 0,\ \alpha\in\R).
\]
\end{prop}
\begin{proof}
Let $n\in\N$ and $V_n := V\vee(-n)$; then the assumption implies that
\[
  \|\rho^{-\alpha}T_{V_n}(t)\rho^\alpha\|_{2\to2} \leqslant Me^{\alpha^2t+(n+\omega)t}
  \qquad (t\geqslant 0,\ \alpha\in\R).
\]
Moreover, $T_{V_n}$ is contractive because $T_{V_n} \leqslant T_V$, and $T_{V_n}$ is self-adjoint.
Thus $T_{V_n}$ is an analytic semigroup of angle $\frac\pi2$,\, $\|T_{V_n}(z)\|\leqslant 1$ for all $z \in \C_+$.
Now \cite[Prop.~2.3]{vog19} (applied with $\ln\rho$ in place of~$\rho$) yields the estimate $\|\rho^{-\alpha}T_{V_n}(t)\rho^\alpha\|_{2\to2} \leqslant e^{\alpha^2t}$ for all $t\geqslant0$ (independent of~$n\,$!),
and letting $n\to\infty$ we obtain the assertion.
\end{proof}

The next result deals with strong continuity and admissibility for weighted semigroups.

\begin{prop}\label{weighted-scontinuity}
Let $\rho\colon\Omega\to(0,\infty)$ be measurable, and let $T$ be a positive $C_0$-semigroup on $L_2(\mu)$ satisfying
\begin{equation}\label{weighted-assu}
  \|\rho^{-1}T(t)\rho\|_{2\to2} \leqslant Me^{\omega t}
  \qquad (t\geqslant0) 
\end{equation}
for some $M \geqslant 1,\ \omega \in \R$.

Then $\rho^{-1}T(t)\rho$ extends to a bounded operator $T^\rho(t)$ on~$L_2(\mu)$, for all $t\geqslant0$, and the family $(T^\rho(t))_{t\geqslant0}$ thus defined is a $C_0$-semigroup on~$L_2(\mu)$.
Moreover, if $V\colon\Omega\to\R$ is bounded from below and $T$-admissible, then $V$ is $T^\rho$-ad\-mis\-si\-ble and $(T^\rho)_V = (T_V)^\rho$.
\end{prop}
\begin{proof}
For $t\geqslant0$ the operator $\rho^{-1}T(t)\rho$ is defined on the dense subspace
$\dom(\rho) = \{f\in L_2(\mu) \setcol \rho f\in L_2(\mu)\}$ and extends by continuity to a bounded operator $T^\rho(t)$ on~$L_2(\mu)$, by~\eqref{weighted-assu}.
It is easy to see that $T^\rho$ is a semigroup;
we show that it is strongly continuous, arguing as in \cite[proof of Prop.~1]{Vo92}.
If $f\in\dom(\rho)$ and $g\in\dom(\rho^{-1})$, then $t\mapsto \langle T^\rho(t)f,g\rangle$ $(= \langle T(t)\rho f,\rho^{-1}g\rangle)$ is continuous.
The bound~\eqref{weighted-assu} implies that the continuity carries over to all $f,g\in L_2(\mu)$.
Thus, $T^\rho$ is weakly continuous and hence strongly continuous; see \cite[Thm.~I.5.8]{EnNa00}.

For bounded~$V$ the identity $(T^\rho)_V = (T_V)^\rho$ follows from the Trotter product formula stated in~\cite[Exercise~III.5.11]{EnNa00}.
In the case of unbounded $V\geqslant0$ we infer that
\[
  (T^\rho)_V(t) = \slim_{n\to\infty}(T^\rho)_{V\wedge n}(t) =\slim_{n\to\infty}(T_{V\wedge n})^\rho(t) = (T_V)^\rho(t)
  \qquad (t\geqslant0);
\]
then as above, the strong continuity of $T_V$ implies that $(T_V)^\rho$ (and hence $(T^\rho)_V$) is strongly continuous, so $V$ is $T^\rho$-admissible.
The assertion in the general case follows similarly.
\end{proof}

Now we can turn to the proof of Theorem~\ref{weighted-ultracon}.

\begin{proof}[Proof of Theorem~\ref{weighted-ultracon}]
By Proposition~\ref{weighted-ultracon2}, the assumed estimates
$\|T_{pV}(t)\|_{2\to2}\leqslant1$ and
$\|\rho^{-\alpha}T(t)\rho^\alpha\|_{2\to2} \leqslant e^{\alpha^2t}$ $(t>0,\ \alpha\in\R)$
imply that
\begin{equation}\label{miracle}
  \|\rho^{-\alpha}T_{pV}(t)\rho^\alpha\|_{2\to2} \leqslant e^{\alpha^2t}
  \qquad (t>0,\ \alpha\in\R). 
\end{equation}

Now fix $\alpha>0$. Let $n\in\N$ and $V_n := V\vee(-n)$.
Applying Proposition~\ref{weighted-scontinuity} with $\rho^\alpha$ in place of $\rho$, we see that
\[
  T_{\rho,\alpha}(t) := e^{-\alpha^2t}\rho^{-\alpha}T(t)\rho^\alpha
  \qquad (t\geqslant0)
\]
defines a $C_0$-semigroup $T_{\rho,\alpha}$ on $L_2(\mu)$.
Moreover, for $s \in \{1,p\}$ the potential $sV_n$ is $T_{\rho,\alpha}$-admissible, and $(T_{\rho,\alpha})_{sV_n}(t) = e^{-\alpha^2t}\rho^{-\alpha}T_{sV_n}(t)\rho^\alpha$ on~$\dom(\rho^\alpha)$ for all $t\geqslant0$.

Let $q:=2p$ and observe that
\[
  \frac{1}{q'} =  \frac{1/p'}{1} + \frac{1/p}{2}, \qquad
  \frac{1}{q} = \frac{1/p'}{\infty} + \frac{1/p}{2}\,.
\]
Note that $\|(T_{\rho,\alpha})(t)\|_{1\to\infty} \leqslant ct^{-\nu}$ by assumption.
Inequalities $T_{pV_n}(t) \leqslant T_{pV}(t)$ and~\eqref{miracle}
imply that $\|(T_{\rho,\alpha})_{pV_n}(t)\|_{2\to2} \leqslant 1$.
Thus, using Theorem~\ref{thm-interpol} we obtain
\[
  \|(T_{\rho,\alpha})_{V_n}(t)\|_{q'\to q}
  \leqslant \|(T_{\rho,\alpha})(t)\|_{1\to\infty}^{1/p'}\|(T_{\rho,\alpha})_{pV_n}(t)\|_{2\to2}^{1/p}
  \leqslant c^{1/p'} t^{-\nu/p'}
  \qquad (t>0).
\]
Moreover, since $T_{V_n} \leqslant T_V$, the assumptions imply
\[
  \|(T_{\rho,\alpha})_{V_n}(t)\|_{1\to1} \leqslant M, \quad
  \|(T_{\rho,\alpha})_{V_n}(t)\|_{\infty\to\infty} \leqslant Me^{(r-1)\alpha^2t}
  \qquad (t>0).
\]
Thus, applying Proposition~\ref{coulhon-extrapol} to the semigroup $(T_{\rho,\alpha})_{V_n}$ we obtain
\[
  \|e^{-\alpha^2t}\rho^{-\alpha}T_{V_n}(t)\rho^\alpha\|_{1\to\infty}
  \leqslant \tilde ct^{-\nu}\bigl(1+(r-1)\alpha^2t\bigr)^{\nu/2p}
  \qquad (t>0),
\]
with a constant $\tilde c>0$ depending only on $p,c,\nu$ and~$M$.
Now the assertion follows by taking the limit $n\to\infty$.
\end{proof}

\subsection{Kernel estimates for Schrödinger semigroups with exponentially weighted $L_1$-bounds}\label{sec-kernel-bounds-exponential}

In this subsection we show the kernel estimate~\eqref{kernel-estimate-exponential}. 
Throughout let $\Omega \subseteq \R^d$ be open, $\DL$ the Dirichlet Laplacian on $\Omega$ and $T$ the generated $C_0$-semigroup on $L_2(\Omega)$.
It is well-known that $T$ is dominated by the free heat semigroup $T^{\R^d}$ on $L_2(\R^d)$, i.e., $T(t) (\ind_{\Omega}f) \leqslant T^{\R^d}(t)f$ for all $t \geqslant 0$ and all $0 \leqslant f \in L_2(\R^d)$; see, e.g., \cite[Prop.~4.23]{ouh05}.
For $\xi \in \R^d$ let $\rho_{\xi}\colon\Omega \rightarrow (0,\infty)$ be defined by $\rho_{\xi}(x):=e^{\xi \cdot x} \ (x \in \Omega)$.
Let $V\colon\Omega \rightarrow \R$ be measurable and assume that there exists $\alpha \in (0,1)$ such that 
\begin{equation}\label{form-small-dirichlet}
\int_{\Omega} |V||u|^2 \dx \leqslant \alpha \langle -\DL u,u \rangle \qquad (u \in \dom(\DL)); 
\end{equation}
note that then $V$ is $T$-admissible by Proposition~\ref{admissibility}.

Further, let $\xi \in \R^d$, and for $t > 0$ let $k_t$ be the integral kernel of $T(t)$.
Then the integral kernel $k_{\xi,t}$ of $\rho_{\xi}T(t)\rho_{\xi}^{-1}$ is given by $k_{\xi,t}(x,y) = e^{\xi \cdot (x - y)} k_t(x,y)$, and since $T$ is dominated by the free heat semigroup, we obtain
\begin{equation}\label{kxit-est}
k_{\xi,t}(x,y) \leqslant (4\pi t)^{-d/2} e^{\xi \cdot (x - y) - |x-y|^2/4t} = (4\pi t)^{-d/2} e^{|\xi|^2 t -|x-y-2t\xi|^2/4t} \qquad (x,y \in \Omega).
\end{equation}
It easily follows that
\begin{equation}\label{Txi-est}
\|\rho_{\xi}T(t)\rho_{\xi}^{-1}\|_{p \rightarrow p} \leqslant e^{|\xi|^2t} \qquad \bigl(t \geqslant 0,\ p \in [1,\infty]\bigr).
\end{equation}

Now we can prove the kernel estimate \eqref{kernel-estimate-exponential}.

\begin{thm}\label{kernel-bound-exponential}
Let \eqref{form-small-dirichlet} be true, and assume that there exists $M \geqslant 1$ such that 
\[
\|\rho_{\xi}T_V(t)\rho_{\xi}^{-1}\|_{1 \rightarrow 1} \leqslant Me^{|\xi|^2t} \qquad (t \geqslant 0,\ \xi \in \R^d). 
\]
Then for every $t > 0$ the operator $T_V(t)$ has an integral kernel $k_t^V \in L_{\infty}(\Omega \times \Omega)$ such that
\[
0 \leqslant k_t^V(x,y) \leqslant ct^{-d/2} e^{-|x-y|^2/4t} \qquad (\text{a.e. } x,y \in \Omega)
\]
for some $c > 0$ independent of $t$.
\end{thm}
\begin{proof}
Without loss of generality let $\K=\C$.

It follows from \eqref{kxit-est} that
\[
\|\rho_{\xi}T(t)\rho_{\xi}^{-1}\|_{1 \rightarrow \infty} \leqslant (4 \pi t)^{-d/2} e^{|\xi|^2t} \qquad (t > 0,\ \xi \in \R^d);
\]
Expressed differently,
\[
\|\rho_{\xi}^{-\beta}T(t)\rho_{\xi}^{\beta} \|_{1 \rightarrow \infty} \leqslant (4 \pi t)^{-d/2} e^{\beta^2t} \qquad (t > 0,\ \beta \in \R)
\]
for all $\xi \in \R^d$ with $|\xi| = 1$. Moreover, by \eqref{Txi-est} we have
\[
\|\rho_{\xi}^{-\beta}T(t)\rho_{\xi}^{\beta} \|_{2 \rightarrow 2} \leqslant e^{\beta^2t} \qquad (t \geqslant 0,\ \beta \in \R)
\]
for all $\xi \in \R^d$ with $|\xi| = 1$.

Let $p := 1/\alpha$ ($>1$). Then, by \eqref{form-small-dirichlet} and Proposition~\ref{admissibility}, the potential $pV$ is $T$-admissible and $T_{pV}$ is contractive.
Since
\[
\|\rho_{\xi}^{-\beta}T_V(t)\rho_{\xi}^{\beta} \|_{1 \rightarrow 1} \leqslant M e^{\beta^2t} \qquad (t \geqslant 0, \ \beta \in \R)
\]
for all $\xi \in \R^d$ with $|\xi| = 1$ by hypothesis and thus also 
\[
\|\rho_{\xi}^{-\beta}T_V(t)\rho_{\xi}^{\beta}\|_{\infty \to \infty} \leqslant Me^{\beta^2 t} \qquad (t \geqslant 0, \ \beta \in \R)
\] 
by duality and the self-adjointness of $T_V$, we can now apply Theorem~\ref{weighted-ultracon} (with $r=1$) and conclude that there exists $c > 0$ such that  
\[
\|\rho_{\xi}^{-\beta}T_V(t) \rho_{\xi}^{\beta} \|_{1 \rightarrow \infty} \leqslant ct^{-d/2}e^{\beta^2t} \qquad (t > 0,\ \beta \in \R)
\]
for all $\xi \in \R^d$ with $|\xi| = 1$, i.e.,
$\|\rho_{\xi}T_V(t)\rho_{\xi}^{-1}\|_{1 \rightarrow \infty} \leqslant ct^{-d/2}e^{|\xi|^2t}$ for all $t > 0$ and all $\xi \in \R^d$. Hence, the Dunford-Pettis theorem implies that for every $t > 0$ the operator $T_V(t)$ has an integral kernel $k_t^V \in L_{\infty}(\Omega \times \Omega)$ such that
\[
0 \leqslant k_t^V(x,y) \leqslant ct^{-d/2} e^{-\xi \cdot (x-y)} e^{|\xi|^2t} \qquad (\text{a.e. } x,y \in \Omega)
\]
for all $\xi \in \R^d$. The assertion now follows from the well-known Davies trick by putting $\xi = (x-y)/2t$.
\end{proof}

\section{Kernel estimates for Schrödinger semigroups: the boundary term}\label{ch-kernel-bounds-phragmen}

In this section we prove the kernel estimate~\eqref{kernel-estimate-double}.
We apply Theorem~\ref{weighted-ultracon} to a weighted Schrödinger semigroup on the positive half-space provided with the weighted measure $x_1^2\dx$ and then use Davies trick. In order to be able to apply Theorem~\ref{weighted-ultracon}, we need to show an exponentially weighted ultracontractivity estimate for this semigroup, which is topic of Subsection~\ref{sec-kernel-bound-double}. The kernel estimate~\eqref{kernel-estimate-double} is also shown in Subsection~\ref{sec-kernel-bound-double}.

\subsection{$L_1$-contractivity and ultracontractivity of the heat semigroup on the positive half-space with weight $x_1$}\label{ch-ultra-stoch-heat-semi}

In this subsection let  
\begin{align*}
\Omega_0:= \begin{cases} 
(0,\infty) \qquad  \qquad &\text{if } d=1, \\
(0,\infty) \times \R^{d-1} &\text{if } d \geqslant 2.
\end{cases}
\end{align*}
Let $T$ be the $C_0$-semigroup on $L_2(\Omega_0)$ generated by the Dirichlet Laplacian on $\Omega_0$, and let $m\colon\Omega_0 \rightarrow (0,\infty)$ be defined by $m(x):= x_1$ ($x \in \Omega_0$).
As in Lemma~\ref{admis-char} we define the unitarily transformed semigroup $T^m$ on $L_2(\Omega_0,m^2\lambda^d)$ by $T^m(t)f := m^{-1}T(t)mf$, where $\lambda^d$ is the Lebesgue measure on $\Omega_0$. We show that the semigroup $T^m$ satisfies exponentially weighted $L_1$- and ultracontractivity estimates on $(\Omega_0,m^2\lambda^d)$.
Later, in Section~\ref{ch-estimates-exponential-double}, we will use these estimates to show corresponding estimates for the perturbed semigroup $(T^m)_V$, in dimension $d=1$.

We will use the fact that for every $t > 0$ the operator $T(t)$ has the integral kernel $k_t\colon\Omega_0 \times \Omega_0 \rightarrow (0,\infty)$ defined by
\[
k_t(x,y):= (4\pi t)^{-d/2}(e^{-|x-y|^2/4t} - e^{-|Sx-y|^2/4t}) \qquad (x,y \in \Omega_0),
\]
where $S:\Omega_0 \to \R^d$ is defined by $Sx:= (-x_1,x_2,\dots,x_d)$. This can be seen by a reflection principle. Observe that 
\begin{equation}\label{halfspace-kernel}
k_t(x,y)= (4\pi t)^{-d/2}e^{-|x-y|^2/4t} \left(1 - e^{-x_1y_1/t} \right) \qquad (x,y \in \Omega_0).
\end{equation}
Thus, by the elementary inequality $\frac12(1\wedge r) \leqslant 1 - e^{-r} \leqslant 1\wedge r$ ($r\geqslant0$), we have
\begin{equation}\label{kernel-est-0}
\tfrac12 \bigl(1 \wedge \tfrac{x_1y_1}{t} \bigr) (4\pi t)^{-d/2}e^{-|x-y|^2/4t} \leqslant k_t(x,y) \leqslant \bigl(1 \wedge \tfrac{x_1y_1}{t} \bigr) (4\pi t)^{-d/2}e^{-|x-y|^2/4t}
\end{equation}
for all $x,y \in \Omega_0$.

To show that $T^m$ satisfies exponentially weighted $L_1$-estimates, we will use the following properties of the kernel~$k_t$.

\begin{lem}\label{int-kernel-heat-semi} 
\rma Let $t > 0$ and $y \in \Omega_0$. Then
\[
\int_{\Omega_0}\frac{x_1}{y_1} k_t(x,y) \dx = 1.
\]

\noindent \rmb Let $t > 0$, $y \in \Omega_0$ and $\xi \in \R^d$. Then
\[
\int_{\Omega_0} \frac{x_1}{y_1} e^{\xi \cdot (x-y)} k_t(x,y) \dx \leqslant 2^{d/2+1}e^{2|\xi|^2t}.
\]
Moreover, if $\xi_1\leqslant0$, then
\[
\int_{\Omega_0} \frac{x_1}{y_1} e^{\xi \cdot (x-y)} k_t(x,y) \dx \leqslant e^{|\xi|^2t}.
\]
\end{lem}

\begin{proof}
(a) Note that in the case $d > 1$ we have
\begin{align*}
&\int_{\Omega_0}\frac{x_1}{y_1} k_t(x,y) \dx \\
&= \int_0^{\infty} \frac{x_1}{y_1} (4 \pi t)^{-1/2} e^{-(x_1 - y_1)^2/4t} \bigl(1- e^{- x_1 y_1/t} \bigr) \int_{\R^{d-1}} (4 \pi t)^{-(d-1)/2} e^{-|\widehat{x} - \widehat{y}|^2/4t} \, d\widehat{x} \,dx_1 \\ 
&= \int_0^{\infty} \frac{x_1}{y_1} (4 \pi t)^{-1/2} e^{-(x_1 - y_1)^2/4t} \bigl(1- e^{- x_1 y_1/t} \bigr) \,dx_1
\end{align*}
for all $t > 0$ and all $y = (y_1, \widehat{y}) \in (0,\infty) \times
\R^{d-1} = \Omega_0$. Hence, we may assume that $d=1$.

Now let $t,y\in(0,\infty)$. Then
\[
\int_0^{\infty} \frac{x}{y} k_t(x,y) \,dx 
=(4 \pi t)^{- 1/2}y^{-1} \bigg( \int_0^{\infty} x e^{-(x-y)^2/4t} \,dx - \int_0^{\infty} x e^{-(-x-y)^2/4t} \,dx \bigg).
\]
Since
\[
  - \int_0^{\infty} x e^{-(-x-y)^2/4t} \,dx = \int_{-\infty}^0 x e^{-(x-y)^2/4t} \,dx,
\]
we conclude that
\begin{align*}
\int_0^{\infty} \frac{x}{y} k_t(x,y) \,dx
 &= (4 \pi t)^{- 1/2}y^{-1} \int_{-\infty}^\infty x e^{-(x-y)^2/4t} \,dx
 = (4 \pi t)^{- 1/2}y^{-1} \int_{-\infty}^\infty (z+y) e^{-z^2/4t} \,dz \\
 &= (4 \pi t)^{- 1/2} \int_{-\infty}^{\infty} e^{-z^2/4t} \,dz = 1. 
\end{align*}

(b) We observe that the Peter-Paul inequality (i.e. $ab \leqslant \frac{1}{2}(\gamma a^2 + \gamma^{-1} b^2)$ for all $a,b \in \R$, $\gamma > 0$) implies
\[
\xi \cdot (x-y) \leqslant 2|\xi|^2t + \frac{1}{8}\frac{|x-y|^2}{t} \qquad (x \in \Omega_0).
\]
Moreover, since the function $[0,\infty) \ni r \mapsto 1-e^{-r} \in [0,\infty)$ is concave and takes the value~$0$ for $r = 0$, we have that 
\[
1-e^{-r} \geqslant \frac12(1-e^{-2r}) \qquad (r\geqslant0).
\]
Using these inequalities we estimate
\begin{align*}
&(4 \pi t)^{-d/2} \int_{\Omega_0}\frac{x_1}{y_1} e^{\xi \cdot (x-y)} e^{-|x-y|^2/4t}(1-e^{-x_1y_1/t}) \dx \\ 
&\qquad\leqslant 2^{d/2} (8 \pi t)^{-d/2} e^{2|\xi|^2t} \int_{\Omega_0}\frac{x_1}{y_1} e^{-|x-y|^2/8t}\,2(1-e^{-x_1y_1/2t}) \dx \\
&\qquad= 2^{d/2 + 1} e^{2|\xi|^2t} \int_{\Omega_0}\frac{x_1}{y_1} k_{2t}(x,y) \dx.
\end{align*}
This completes the proof of the first assertion of~(b), by part (a).

For the proof of the second assertion we put
\[
G(\xi):= (4\pi t)^{-d/2}\int_{\Omega_0} \frac{x_1}{y_1} e^{-|x-y|^2/4t + \xi\cdot(x-y) -|\xi|^2t} (1 - e^{-x_1y_1/t}) \dx \qquad (\xi \in \R^d);
\]
we have to show that $G(\xi) \leqslant 1$ if $\xi_1 \leqslant 0$.
Let $F \colon \R^d \to [0,\infty)$ be defined by $F(x):= (4\pi t)^{-d/2}\frac{x_1}{y_1}(1-e^{-x_1y_1/t})$ if $x \in \Omega_0$ and $F(x):= 0$ if $x \in \R \setminus \Omega_0$. Then, using the identity
\begin{equation}\label{binomi}
-\frac{1}{4t}|x-y|^2 + \xi\cdot(x-y) -|\xi|^2t = -\frac{1}{4t}|x-y-2\xi t|^2
\end{equation}
and a change of variables, we can rewrite $G(\xi)$ as
\[
G(\xi) = \int_{\Omega_0} F(x) e^{-|x-y-2\xi t|^2/4t} \dx = \int_{\R^d} F(z + y + 2\xi t) e^{-|z|^2/4t} \,dz \qquad (\xi \in \R^d).
\]
Therefore, since $F$ is monotone increasing in the $x_1$-variable and $F(x) = F(x_1,0,\dots,0)$ for all $x \in \R^d$, $G\colon \R^d \to [0,\infty)$ is increasing in the $\xi_1$-variable and $G(\xi)=G(\xi_1,0,\dots,0)$ for all $\xi \in \R^d$, which in turn implies that $G(\xi) \leqslant G(0)$ for all $\xi \in \R^d$, $\xi_1 \leqslant 0$.
Now $G(0) = \int_{\Omega_0} \frac{x_1}{y_1} k_t(x,y)\dx = 1$ by part~(a), so we conclude that the asserted estimate $G(\xi) \leqslant 1$ holds for $\xi_1 \leqslant 0$.
\end{proof}

Now we turn to the proof of the exponentially weighted $L_1$- and ultracontractivity estimates for~$T^m$.

\begin{prop}\label{stoch-ultra-double-weighted}
\rma Let $\xi \in \R^d$ and $t \geqslant 0$. Then
\[
\|m^{-1}\rho_{\xi}T(t)\rho_{\xi}^{-1}m\|_{L_1(m^2\lambda^d) \to L_1(m^2\lambda^d)} \leqslant 2^{d/2 + 1}e^{2|\xi|^2t}.
\]
Moreover, if $\xi_1 \leqslant 0$, then 
\[
\|m^{-1}\rho_{\xi}T(t)\rho_{\xi}^{-1}m\|_{L_1(m^2\lambda^d) \rightarrow L_1(m^2\lambda^d)} \leqslant e^{|\xi|^2t}.
\]
\noindent \rmb Let $\xi \in \R^d$ and $t > 0$. Then 
\[
\|m^{-1}\rho_{\xi}T(t)\rho_{\xi}^{-1}m \|_{L_1(m^2\lambda^d) \rightarrow L_\infty(m^2\lambda^d)} \leqslant (4 \pi)^{- d/2} t^{-(d/2 +1)}e^{|\xi|^2t}.
\]
\end{prop}
\begin{proof}
(a)
For $t=0$ there is nothing to show, so we assume that $t > 0$. Then $k_t$ is the integral kernel of $T(t)$; hence for $f \in L_1(m^2\lambda^d)$ with $\rho_{\xi}^{-1}mf \in L_2(\Omega_0)$ we obtain
\begin{align}
\|m^{-1}\rho_{\xi}T(t)\rho_{\xi}^{-1}mf\|_{L_1(m^2\lambda^d)} &\leqslant \int_{\Omega_0} \int_{\Omega_0} x_1^{-1}e^{\xi \cdot x}k_t(x,y) e^{-\xi \cdot y}y_1 |f(y)| \,dy \,d(m^2\lambda^d)(x) \nonumber \\
&= \int_{\Omega_0} \int_{\Omega_0} \frac{x_1}{y_1} e^{\xi \cdot (x-y)} k_t(x,y) \,dx \,f(y) \,d(m^2\lambda^d)(y). \label{norm-heat-se-simply-weighted}
\end{align}
Consequently, the assertions of (a) follow from Lemma~\ref{int-kernel-heat-semi}(b).

(b) Let $f \in L_1(m^2\lambda^d)$ be such that $\rho_{\xi}^{-1}mf \in L_2(\Omega_0)$. Then we estimate, using \eqref{kernel-est-0} for the first and \eqref{binomi} for the second inequality, 
\begin{align*}
|m^{-1}\rho_{\xi}T(t)\rho_{\xi}^{-1}m f(x)| &\leqslant \int_{\Omega_0} \frac{1}{x_1 y_1} \cdot \frac{x_1 y_1}{t} (4 \pi t)^{-d/2} e^{\xi \cdot (x-y)}e^{-|x-y|^2/4t} |f(y)| \, d(m^2 \lambda^d)(y)  \\ 
&\leqslant (4 \pi)^{- d/2} t^{- (d/2 + 1)} e^{|\xi|^2t}\|f\|_{L_1(m^2 \lambda^d)} \qquad (\text{a.e. } x \in \Omega_0). \tag*{\qedhere}
\end{align*}
\end{proof}

\begin{rem}\label{stoch-ultra-double-weighted-rem}
Note that Proposition~\ref{stoch-ultra-double-weighted}(a) with $\xi = 0$ implies that $T^m$ is $L_1$-contractive. Actually, it follows from the identity \eqref{norm-heat-se-simply-weighted} and Lemma~\ref{int-kernel-heat-semi}(a) that $T^m$ is even stochastic, i.e.\ $\|T^m(t)f\|_{L_1(m^2 \lambda^d)} = \|f\|_{L_1(m^2 \lambda^d)}$ for all $f \in L_1\cap L_2(m^2 \lambda^d)_+$.
\end{rem}

\subsection{Kernel estimates for Schrödinger semigroups with weighted $L_1$-bounds on the positive half-space}\label{sec-kernel-bound-double}

Throughout this subsection let $\Omega_0$, $m$ and $\rho_{\xi}$ be defined as in Subsection~\ref{ch-ultra-stoch-heat-semi},
and let $T$ be the $C_0$-semigroup on $L_2(\Omega_0)$
generated by the Dirichlet Laplacian $\DL$ on~$\Omega_0$.
Let $V\colon\Omega \rightarrow \R$ be measurable, and assume that $V$ satisfies the form smallness condition~\eqref{form-small-dirichlet} for some $\alpha \in (0,1)$, with $\Omega = \Omega_0$.
We will now show the kernel estimate 
\eqref{kernel-estimate-double} and then derive the kernel estimate \eqref{kernel-estimate-exponential-double}.

\begin{thm}\label{ultra-Schr-semi-double} 
Let \eqref{form-small-dirichlet} be true, and assume that there exists $M \geqslant 1$ such that 
\begin{align*}
&\|m^{-1}\rho_{\xi}T_V(t)\rho_{\xi}^{-1}m\|_{L_1(m^2\lambda^d) \rightarrow L_1(m^2\lambda^d)} \leqslant Me^{2|\xi|^2t} \qquad (t \geqslant 0, \ \xi \in \R^d, \ \xi_1 \geqslant 0), \\
&\|m^{-1}\rho_{\xi}T_V(t)\rho_{\xi}^{-1}m\|_{L_1(m^2\lambda^d) \rightarrow L_1(m^2\lambda^d)} \leqslant Me^{|\xi|^2t} \qquad (t \geqslant 0, \ \xi \in \R^d, \ \xi_1 \leqslant 0).
\end{align*}
Then for every $t > 0$ the operator $T_V(t)$ has an integral kernel $k_t^V \in L_{\infty}(\Omega \times \Omega)$ such that there exists $c > 0$ independent of $t$ with 
\[
0 \leqslant k_t^V(x,y) \leqslant cx_1y_1t^{-(d/2 + 1)} \smash{\left(1 + \frac{|x-y|^2}{4t} \right)^{\alpha(d+2)/4}} e^{-|x-y|^2/4t} \vphantom{\Big|}
\]  
for a.e.\ $x,y \in \Omega_0$.
\end{thm}

For the proof of Theorem~\ref{ultra-Schr-semi-double} we define the unitarily transformed semigroup $T_V^m$ on $L_2(\Omega_0,m^2\lambda^d)$ by $T_V^m(t)f := m^{-1}T_V(t)mf$ ($t \geqslant 0$, $f \in L_2(m^2\lambda^d)$).
By Lemma~\ref{admis-char}, $T_V^m$ is a $C_0$-semigroup and $V$ is $T^m$-admissible with $(T^m)_V = T_V^m$.
(Recall from Proposition~\ref{admissibility} that $V$ is $T$-admissible.)

\begin{proof}[Proof of Theorem~\ref{ultra-Schr-semi-double}]
Without loss of generality let $\K=\C$.

Let $\xi \in \R^d$, $|\xi| = 1$. By Proposition~\ref{stoch-ultra-double-weighted}(b) we have
\[
\|\rho_{\xi}^{-\beta}T^m(t)\rho_{\xi}^{\beta}\|_{L_1(m^2\lambda^d) \rightarrow L_{\infty}(m^2\lambda^d)} \leqslant (4\pi)^{-d/2}t^{-(d/2 + 1)}e^{\beta^2t}  \qquad (t > 0, \ \beta \in \R).
\]
From \eqref{Txi-est} we deduce that
\[
\|\rho_{\xi}^{-\beta}T^m(t)\rho_{\xi}^{\beta}\|_{L_2(m^2\lambda^d) \rightarrow L_2(m^2\lambda^d)} \leqslant e^{\beta^2t}  \qquad (t \geqslant 0, \ \beta \in \R)
\]
since $L_2(m^2\lambda) \ni f \mapsto mf \in L_2(0,\infty)$ is an isometric isomorphism.
Moreover, the hypotheses imply that
\[
\|\rho_{\xi}^{-\beta}(T^m)_V(t)\rho_{\xi}^{\beta}\|_{L_1(m^2\lambda^d) \to L_1(m^2\lambda^d)} \leqslant Me^{\beta^2 t} \qquad (t \geqslant 0, \ \beta > 0)
\]
if $\xi_1 \geqslant 0$, and
\[
\|\rho_{\xi}^{-\beta}(T^m)_V(t)\rho_{\xi}^{\beta}\|_{L_1(m^2\lambda^d) \to L_1(m^2\lambda^d)} \leqslant Me^{2\beta^2 t} \qquad (t \geqslant 0, \ \beta > 0)
\]
if $\xi_1 \leqslant 0$; by duality and self-adjointness of $(T^m)_V$ the latter is equivalent to
\[
\|\rho_{\xi}^{-\beta}(T^m)_V(t)\rho_{\xi}^{\beta}\|_{L_{\infty}(m^2\lambda^d) \to L_{\infty}(m^2\lambda^d)} \leqslant Me^{2\beta^2 t} \qquad (t \geqslant 0, \ \beta > 0)
\]
for $\xi_1 \geqslant 0$. 

Let $p:= 1/\alpha$ ($> 1$). Then $pV$ is $T^m$-admissible and $(T^m)_{pV}$ is contractive, by Proposition~\ref{admissibility} and Lemma~\ref{admis-char}. Hence, Theorem~\ref{weighted-ultracon} implies that there exists $c > 0$ such that 
\[
\|\rho_{\xi}^{-\beta}T^m_V(t)\rho_{\xi}^{\beta}\|_{L_1(m^2\lambda^d) \rightarrow L_{\infty}(m^2\lambda^d)} \leqslant c t^{- (d/2 + 1)}(1 + \beta^2t)^{(d/2 + 1)/2p}e^{\beta^2t} \qquad (t > 0, \ \beta \in \R),
\]
for all $x \in \R^d$, $\xi_1 \geqslant 0$, $|\xi| = 1$, i.e.,
\[
\|\rho_{\xi}T^m_V(t)\rho_{\xi}^{-1}\|_{L_1(m^2\lambda^d) \rightarrow L_{\infty}(m^2\lambda^d)} \leqslant c t^{- (d/2 + 1)}(1 + |\xi|^2t)^{\alpha(d + 2)/4}e^{|\xi|^2t} \qquad (t > 0, \ \xi \in \R^d).
\]
Thus, by the Dunford-Pettis theorem one concludes that for every $t > 0$ the operator $(T^m)_V(t)$ has an integral kernel $k_t^{V,m} \in L_{\infty}(\Omega_0 \times \Omega_0)$ such that
\begin{equation}\label{m-xi-kernel}
0 \leqslant k_t^{V,m}(x,y) \leqslant ct^{-(d/2 +1)}(1 + |\xi|^2t)^{\alpha(d+2)/4} e^{|\xi|^2t}e^{-\xi \cdot (x-y)} \qquad (\text{a.e. } x,y \in \Omega_0) 
\end{equation}
for all $\xi \in \R^d$. Now one easily shows that 
\[
k_t^V(x,y):= x_1 k_t^{V,m}(x,y)y_1 \qquad (x,y \in \Omega_0)
\]
defines an integral kernel $k_t^V$ of $T_V(t)$, and with \eqref{m-xi-kernel} and Davies trick one sees that this integral kernel satisfies the asserted kernel estimate.
\end{proof}

\begin{rem}
If one could show that 
\begin{equation}\label{l1-estimate-exp-double}
\|m^{-1}\rho_{\xi}T_V(t)\rho_{\xi}^{-1}m\|_{L_1(m^2\lambda^d) \to L_1(m^2\lambda^d)} \leqslant Me^{|\xi|^2t} \qquad (t \geqslant 0)
\end{equation}
also holds for $\xi_1>0$, then with the same proof as above one could infer that $k_t^V(x,y) \leqslant c x_1y_1t^{-(d/2 + 1)} e^{-|x-y|^2/4t}$ for some $c>0$. Thus one would obtain a kernel estimate better than~\eqref{kernel-estimate-double},
without the polynomial correction factor $\bigl(1 + (x-y)^2/4t\bigr)\rule{0pt}{1.6ex}^{\alpha(d+2)/4}$.
Unfortunately, if $\xi_1>0$, then one can show that \eqref{l1-estimate-exp-double} already fails for $V=0$ (cf.\ the proof of Lemma~\ref{int-kernel-heat-semi}(b)).
\end{rem}

As a direct consequence of  Theorems~\ref{kernel-bound-exponential} and~\ref{ultra-Schr-semi-double} we now obtain the kernel estimate~\eqref{kernel-estimate-exponential-double}.

\begin{cor}\label{kernel-bound-Schr-semi}
Let \eqref{form-small-dirichlet} be true, and assume that there exists $M \geqslant 1$ such that 
\begin{align*}
\|\rho_{\xi}T_V(t)\rho_{\xi}^{-1}\|_{L_1(\Omega_0) \rightarrow L_1(\Omega_0)} &\leqslant M e^{|\xi|^2t} \qquad (t \geqslant 0,\ \xi \in \R^d), \\
\|m^{-1}\rho_{\xi}T_V(t)\rho_{\xi}^{-1}m\|_{L_1(m^2\lambda^d) \rightarrow L_1(m^2\lambda^d)} &\leqslant Me^{2|\xi|^2t} \qquad (t \geqslant 0, \ \xi \in \R^d, \ \xi_1 \geqslant 0), \\
\|m^{-1}\rho_{\xi}T_V(t)\rho_{\xi}^{-1}m\|_{L_1(m^2\lambda^d) \rightarrow L_1(m^2\lambda^d)} &\leqslant Me^{|\xi|^2t} \qquad (t \geqslant 0, \ \xi \in \R^d, \ \xi_1 \leqslant 0).
\end{align*}
Then for every $t>0$ the operator $T_V(t)$ has an integral kernel $k_t^{V} \in L_{\infty}(\Omega_0 \times \Omega_0)$, and there exists $c > 0$ independent of~$t$ such that 
\[
0 \leqslant k_t^{V}(x,y) \leqslant c \left(1 \wedge \left( \frac{x_1y_1}{t} \left(1 + \frac{|x-y|^2}{4t} \right)^{\alpha(d+2)/4} \right) \right) t^{-d/2}e^{-|x-y|^2/4t} 
\]
for a.e.\ $x,y \in \Omega_0$.
\end{cor}

\section{$L_1$-estimates for weighted Schrödinger semigroups on the positive real axis}\label{ch-estimates-exponential-double}

In this section we prove Theorem~\ref{thm-main} by showing that the Schrödinger semigroup~$T_V$ on $L_2(0,\infty)$ satisfies the three weighted $L_1$-estimates assumed in Corollary~\ref{kernel-bound-Schr-semi} if the potential~$V$ satisfies the integral condition~\eqref{int-cond}.
The basic tool for the proof of the three $L_1$-estimates will be given in Subsection~\ref{sec-estimates-perturbed}.
In Subsection~\ref{sec-kernel-resolvent} we show a kernel estimate for the resolvents of the Dirichlet Laplacian on $(0,\infty)$.
This kernel estimate is used in Subsection~\ref{sec-estimates-exponential-double} to show that the operator norm $\|V(\lambda - A)^{-1}\|_{1\to1}$ is small, where $A$ is the Dirichlet Laplacian provided with the weights that are needed for the application of Corollary~\ref{kernel-bound-Schr-semi};
the smallness of the operator norm will enable us to apply the results of Subsection~\ref{sec-estimates-perturbed} to infer the three weighted $L_1$-estimates indicated above.

\subsection{$L_1$-estimates for perturbed $C_0$-semigroups}\label{sec-estimates-perturbed}

The basis for our $L_1$-estimates is the following general result for absorption semigroups on~$L_1$.

\begin{prop}\label{estimates-perturbed-neu}
Let $(\Omega,\mu)$ be a measure space. Let $T$ be a positive $C_0$-semigroup on $L_1(\mu)$ with generator $A$, and let $M \geqslant 1$, $\omega \in \R$ be such that  
\[
\|T(t)\|_{1 \rightarrow 1} \leqslant M e^{\omega t} \qquad (t \geqslant 0).
\]
Let $V\colon\Omega \rightarrow \R$ be measurable, and assume that there exist $\lambda > \omega,\ \alpha \in (0,1)$ such that 
\begin{equation}\label{Miyadera-Voigt-cond-2}
\|V(\lambda - A)^{-1}\|_{1 \rightarrow 1} \leqslant \alpha.
\end{equation}
Then $V$ is $T$-admissible,
\[
\|T_V(t)\|_{1 \rightarrow 1} \leqslant \frac{M}{1-\alpha}e^{\lambda t} \qquad (t \geqslant 0).
\]
\end{prop}
\begin{proof}
By rescaling we can assume without loss of generality that $\omega = 0$.
For $n \in \N$ let $V_n:= (V \wedge n) \vee (-n) \ (\in L_{\infty}(\mu))$. 

(i) In the first step we show that 
\begin{equation}\label{bounded-L1-estimate}
\|e^{-\lambda t}T_{V_n}(t)\|_{1 \rightarrow 1} \leqslant \frac{M}{1-\alpha} \qquad (t \geqslant 0, \ n \in \N).  \end{equation}
For this, fix $n \in \N$ and note that $\|V_n(\lambda - A)^{-1}\|_{1 \rightarrow 1} \leqslant \alpha$ by hypothesis. Since $L_1(\mu) \ni f \mapsto \int_{\Omega} |V_n| f \,d\mu \in \K$ is a continuous linear functional, we have  
\[
\int_0^{\infty} \|V_n e^{-\lambda t} T(t)f\|_1 \,dt
= \left\|V_n \int_0^\infty e^{-\lambda t}T(t)f \,dt\right\|_1
= \|V_n(\lambda - A)^{-1}f\|_1 \leqslant \alpha \|f\|_1
\]
for all $0 \leqslant f \in L_1(\mu)$, so by the positivity of $T$ it follows that 
\[
\int_0^{\infty} \|V_n e^{-\lambda t} T(t)f\|_1 \,dt \leqslant \alpha \|f\|_1 \qquad (f \in L_1(\mu)).
\]
Thus, \cite[Thm.~1(c)]{Vo77} (applied with $B=-V_n$) implies~\eqref{bounded-L1-estimate}. 

(ii) Since~\eqref{Miyadera-Voigt-cond-2} holds with $-V^-$ and $-V^+$ in place of $V$, we conclude from step~(i) and~\cite[Prop.~2.2]{Vo88} that $-V^-$ and $-V^+$ are $T$-admissible. By~\cite[Prop.~3.3(b)]{Vo88} the latter implies that also $V^+$ is $T$-admissible, so that $V$ is $T$-admissible.
Now the asserted estimate follows from~\eqref{bounded-L1-estimate} since $T_{V_n}(t) \rightarrow T_V(t)$ strongly for all $t \geqslant 0$.
\end{proof}

We will apply the above proposition in the form of the next result on absorption semigroups on~$L_2$.

\begin{cor}\label{estimates-perturbed}
Let $(\Omega,\mu)$ be a measure space. Let $T$ be a positive $C_0$-semigroup on $L_2(\mu)$ with generator $A$, and assume that there exist $M \geqslant 1$, $\omega \in \R$ such that
\[
\|T(t)\|_{1 \rightarrow 1} \leqslant M e^{\omega t}, \qquad \|T(t)\|_{2 \rightarrow 2} \leqslant M e^{\omega t}.
\]
for all $t \geqslant 0$. Let $V\colon\Omega \rightarrow \R$ be $T$-admissible, and assume that there exist a sequence $(\lambda_k)_{k \in \N}$ in $(\omega ,\infty)$ converging to $\omega$ and $\alpha \in (0,1)$ such that 
\begin{equation}\label{Miyadera-Voigt-cond}
\|V(\lambda_k - A)^{-1}\|_{1 \rightarrow 1} \leqslant \alpha \qquad (k \in \N).
\end{equation}
Then 
\[
\|T_V(t)\|_{1 \rightarrow 1} \leqslant \frac{M}{1-\alpha}e^{\omega t} \qquad (t \geqslant 0).
\]
\end{cor}
\begin{proof}
For every $t \geqslant 0$ the operator $T(t)|_{L_1 \cap L_2(\mu)}$ extends to a bounded operator $T_1(t)$ on $L_1(\mu)$ with $\|T_1(t)\|_{1 \rightarrow 1} \leqslant Me^{\omega t}$,
and the mapping $T_1 \colon [0,\infty) \to \mathcal{L}(L_1(\mu))$ thus defined is a positive $C_0$-semigroup by \cite[Thm.~7]{Vo92}.
Since~\eqref{Miyadera-Voigt-cond} holds for $V_n:=(V \wedge n) \vee (-n)$ in place of $V$, it follows from Proposition~\ref{estimates-perturbed-neu} that 
\[
\|(T_1)_{V_n}(t)\|_{1 \rightarrow 1} \leqslant \frac{M}{1 - \alpha} e^{\lambda_k t} \qquad (t \geqslant 0,\ k \in \N)
\]
for all $n \in \N$. Letting $k \to \infty$ we obtain 
\[
\|(T_1)_{V_n}(t)\|_{1 \rightarrow 1} \leqslant \frac{M}{1 - \alpha} e^{\omega t} \qquad (t \geqslant 0, \ n \in \N).
\]
Since $(T_1)_{V_n}(t)|_{L_1 \cap L_2(\mu)} = T_{V_n}(t)|_{L_1 \cap L_2(\mu)}$ for all $n \in \N$ by \cite[Prop.~3.1(a)]{Vo86},
the assertion follows by letting $n\to\infty$ in the last inequality.
\end{proof}

The final result of this subsection, a version of \cite[Proposition~4.6]{Vo86}, will be needed for the proof of Theorem~\ref{thm-necessary}, a kind of converse to Theorem~\ref{thm-main}. It is more or less known; we include a proof for the reader's convenience.

\begin{prop}\label{Global-Kato}
Let $(\Omega,\mu)$ be a measure space. Let $T$ be a positive $C_0$-semigroup on~$L_1(\mu)$, and let $V\colon\Omega \rightarrow (-\infty,0]$ be measurable.
Assume $T$ is stochastic, that $V$ is $T$-admissible and that $M := \sup_{t\geqslant0}\|T_V(t)\|_{1 \to 1} < \infty$.
Then $V$ satisfies the global Kato class condition
\[
\int_0^{\infty} \|VT(t)f\|_1 \,dt \leqslant (M-1)\|f\|_1 \qquad (f\in L_1(\mu)).
\]
\end{prop}
\begin{proof}
Since $T$ is a positive semigroup, it suffices to show the estimate for $f\in L_1(\mu)_+$.
Let $t>0$, $n\in\N$ and $V_n := V\vee(-n)$. By Duhamel's formula (see, e.g., \cite[Cor.~III.1.7]{EnNa00}) we have
\[
T_V(t)f \geqslant T_{V_n}(t)f = T(t)f - \int_0^t T(t-s)V_nT_{V_n}(s)f\,ds.
\]
Since $V_n\leqslant0$ and $T$ is stochastic, it follows that
\[
\|T_V(t)f\|_1 - \|f\|_1 \geqslant \int_0^t \|V_nT_{V_n}(s)f\|_1\,ds \geqslant \int_0^t \|V_nT(s)f\|_1\,ds.
\]
Letting $n\to\infty$ and $t\to\infty$ we obtain
$\int_0^\infty \|VT(s)f\|_1\,ds \leqslant (M-1)\|f\|_1$.
\end{proof}

The above argument even shows that $\int_0^{\infty} \|VT_V(t)f\|_1 \,dt \leqslant (M-1)\|f\|_1$, but we will not need this stronger estimate.

\subsection{Integral kernel for the resolvents of the Dirichlet Laplacian on the positive real axis}\label{sec-kernel-resolvent}

In this subsection let $\DL$ be the Dirichlet Laplacian on $(0,\infty)$ and $T$ the generated $C_0$-semigroup on $L_2(0,\infty)$.
We show that for $\lambda > 0$ the integral kernel of the resolvent $(\lambda - \DL)^{-1}$ is given by the Green function $G_{\lambda}\colon (0,\infty) \times (0,\infty) \rightarrow (0,\infty)$,
\[
G_{\lambda}(x,y):= \frac{1}{2\sqrt{\lambda}}e^{-\sqrt{\lambda}|x-y|}\left( 1-e^{-2\sqrt{\lambda}(x \wedge y)} \right).
\]
First note that
$
(\lambda - \DL)^{-1} = \int_0^{\infty}e^{-\lambda t}T(t) \,dt
$
(strong integral), and recall from~\eqref{halfspace-kernel} that for every $t > 0$ the operator $T(t)$ has the integral kernel $k_t\colon(0,\infty) \times (0,\infty) \rightarrow (0,\infty)$ defined by
\begin{equation}\label{kt1}
k_t(x,y):= (4\pi t)^{-1/2}e^{-|x-y|^2/4t} \left(1 - e^{-xy/t} \right) \qquad (x,y \in (0,\infty));
\end{equation}
thus,
\[
(\lambda - \DL)^{-1}f(x) = \int_0^{\infty} \!\! \int_0^{\infty} e^{-\lambda t} k_t(x,y) \,dt f(y) \dy \qquad (\text{a.e. } x \in (0,\infty))
\]
for all $f \in L_2(0,\infty)$ and all $\lambda > 0$.
To conclude that $G_\lambda$ is the integral kernel of $(\lambda - \DL)^{-1}$, we now show that $(0,\infty) \ni \lambda \mapsto G_{\lambda}(x,y) \in (0,\infty)$ is the Laplace transform of the function $(0,\infty) \ni t \mapsto k_t(x,y) \in (0,\infty)$, for every $(x,y) \in (0,\infty) \times (0,\infty)$.

\begin{lem}
Let $\lambda,x,y \in (0,\infty)$. Then
\[
G_{\lambda}(x,y) = \int_0^{\infty} e^{-\lambda t} k_t(x,y) \,dt.
\]
\end{lem}
\begin{proof}
We observe that
\begin{align}
\int_0^{\infty} e^{-\lambda t} k_t(x,y) \,dt &= \int_0^{\infty} e^{-\lambda t} (4 \pi t)^{-1/2} e^{-(x-y)^2/4t} \bigl(1 - e^{-xy/t} \bigr) \, dt \nonumber \\
&= \frac{1}{2 \sqrt{\pi}} \left( \int_0^{\infty} t^{-1/2} e^{-(x-y)^2/4t - \lambda t}\,dt - \int_0^{\infty} t^{-1/2} e^{-(x+y)^2/4t - \lambda t} \,dt \right).  \label{kernel-heat-computation}
\end{align}

To compute the two integrals in the right hand side, let $r>0$, $G(t) := -\int_t^\infty e^{-s^2} \,ds$ ($t \in \R$) and 
\[
F(t) := e^{r\sqrt{\lambda}}\mkern1.5mu\mathopen G\left(\sqrt{\lambda t} + \frac{r}{2 \sqrt{t}}\right)
+ e^{-r\sqrt{\lambda}}\mkern1.5mu\mathopen G\left(\sqrt{\lambda t} - \frac{r}{2 \sqrt{t}}\right) \qquad (t > 0).
\]
Using $G'(t) = e^{-t^2}$ $(t \in \R$), one easily verifies that
\[
F'(t) = \sqrt\lambda \, t^{-1/2} e^{-r^2/4t - \lambda t} \qquad (t >0).
\]
Moreover, since $G(t) \rightarrow 0$ as $t \rightarrow \infty$ and $G(t) \rightarrow -\sqrt\pi$ as $t \rightarrow -\infty$, we have
\[
F(t) \rightarrow 0 \quad (t \rightarrow \infty), \qquad  F(t) \rightarrow -\sqrt{\pi} \, e^{-r\sqrt{\lambda}} \quad (t \rightarrow 0^+).
\]
Hence,
\[
\int_0^{\infty} t^{-1/2} e^{-r^2/4t - \lambda t} \,dt = \frac{1}{\sqrt\lambda} \left( \mkern1mu \lim_{t\to\infty} F(t) - \lim_{t\to0} F(t) \right) = \sqrt{\frac{\pi}{\lambda}}\,e^{-r\sqrt{\lambda}}.
\]
Note that this identity also holds for $r=0$.
Using~\eqref{kernel-heat-computation}, we conclude that 
\begin{align*}
\int_0^{\infty} e^{-\lambda t} k_t(x,y) \,dt &= \frac{1}{2\sqrt{\pi}} \cdot \sqrt{\frac{\pi}{\lambda}} \left( e^{-\sqrt{\lambda}|x-y|} - e^{-\sqrt{\lambda}(x+y)}\right) \\ 
&= \frac{1}{2\sqrt{\lambda}}e^{-\sqrt{\lambda}|x-y|} \left(1-e^{-\sqrt{\lambda}(x+y - |x-y|)} \right)
\\&= \frac{1}{2\sqrt{\lambda}}e^{-\sqrt{\lambda}|x-y|} \left(1-e^{-2\sqrt{\lambda}(x \wedge y)} \right) = G_{\lambda}(x,y). \tag*{\qedhere}
\end{align*}
\end{proof}

Using the elementary inequality $1 - e^{-r} \leqslant r \ (r \geqslant 0)$ we obtain the following estimate for the Green function $G_{\lambda}$ that will be crucial in the next subsection:
\begin{equation}\label{kernel-resolvent-estimate}
G_{\lambda}(x,y) \leqslant (x \wedge y) e^{-\sqrt{\lambda}|x-y|} \qquad \bigl(\lambda>0,\ x,y \in (0,\infty)\bigr).
\end{equation}

\subsection{$L_1$-estimates for weighted Schrö\-dinger semigroups on $(0,\infty)$}\label{sec-estimates-exponential-double}

Let again $\DL$ be the Dirichlet Laplacian on $(0,\infty)$ and $T$ the generated $C_0$-semigroup on $L_2(0,\infty)$.
Let $V\colon(0,\infty) \rightarrow \R$ be measurable and assume that there exists $\alpha \in (0,1)$ such that  
\begin{equation}\label{int-cond2}
\int_0^{\infty} x|V(x)| \dx \leqslant \alpha. 
\end{equation}

Our first goal is to show $L_1$-estimates for the Schrödinger semigroup $T_V$ with exponential weights $\rho_{\xi}\colon(0,\infty) \rightarrow (0,\infty)$ defined by $\rho_{\xi}(x):=e^{\xi \cdot x} \ (x > 0 )$, for all $\xi \in \R$.
We achieve this by first showing that \eqref{int-cond2} implies an estimate of the form \eqref{Miyadera-Voigt-cond}
(see Proposition~\ref{Miyadera-Voigt-cond-exponential}(a) below) and then applying Corollary~\ref{estimates-perturbed} to the semigroup $T^\xi$ defined as follows, for $\xi \in \R$: due to \eqref{Txi-est}, $\rho_{\xi}T(t)\rho_{\xi}^{-1}$ extends to a bounded operator $T^{\xi}(t)$ on~$L_2(0,\infty)$ for all $t\geqslant0$, and by Proposition~\ref{weighted-scontinuity} the family $(T^{\xi}(t))_{t\geqslant0}$ thus defined is a (positive) $C_0$-semigroup on $L_2(0,\infty)$. We denote by $\Delta_{\textnormal D,\xi}$ the generator of $T^{\xi}$.

The second goal is to prove $L_1$-estimates for $T_V$ with exponential weight $\rho_{\xi}$ and weight $m\colon(0,\infty) \rightarrow (0,\infty)$ defined by $m(x):= x \ (x > 0)$ on the measure space $((0,\infty),m^2\lambda)$. We will show this in a manner similar to the first goal from above,
working with the unitarily transformed semigroup $T^{\xi,m}$ on $L_2(\Omega_0,m^2\lambda)$ defined by 
\[
T^{\xi,m}(t)f:=m^{-1}T^{\xi}(t)mf \qquad (t \geqslant 0, \ f \in L_2(m^2\lambda)).
\]
Note that $T^{\xi,m}$ is a positive $C_0$-semigroup satisfying 
\begin{equation}\label{l2-estimate-exp-m}
\|T^{\xi,m}(t)\|_{L_2(m^2\lambda) \to L_2(m^2\lambda)} \leqslant e^{\xi^2t} \qquad (t \geqslant 0)
\end{equation}
since $L_2(m^2\lambda) \ni f \mapsto mf \in L_2(0,\infty)$ is an isometric lattice isomorphism and $T^{\xi}$ is positive and satisfies \eqref{Txi-est}. We denote by $\Delta_{\textnormal D,\xi}^m$ the generator of $T^{\xi,m}$.

\begin{prop}\label{Miyadera-Voigt-cond-exponential}
Let \eqref{int-cond2} be true, and let $\xi \in \R$. Then 

\rma $\|V(\mu - \Delta_{\textnormal D,\xi})^{-1}\|_{1 \rightarrow 1} \leqslant \alpha$ for all $\mu > \xi^2$, and 

\rmb $\|V(\mu - \Delta_{\textnormal D,\xi}^m)^{-1}\|_{L_1(m\lambda^2) \to L_1(m\lambda^2)} \leqslant \alpha$ for all $\mu > \xi^2$.
\end{prop}
\begin{proof}
We assume without loss of generality that $V \in L_{\infty}(0,\infty)$. (This can be done since $V_n:= |V| \wedge n \in L_{\infty}(0,\infty)$ satisfies~\eqref{int-cond2} for every $n \in \N$ and, if the assertions of (a) and~(b) hold with $V_n$ in place of~$V$, then they also hold for $V$ by monotone convergence.)

(a) Let $D := \{L_1\cap L_2(0,\infty) \setcol \rho_\xi^{-1}f\in L_2(0,\infty)\}$, and let $f\in D$.
Then $(\mu - \Delta_{\textnormal D,\xi})^{-1}f = \rho_{\xi}(\mu - \DL)^{-1}\rho_{\xi}^{-1}f$ since $(\mu - \Delta_{\textnormal D,\xi})^{-1}$ is the Laplace transform of $T^{\xi}$.
Thus, using the estimate~\eqref{kernel-resolvent-estimate} for the integral kernel $G_{\mu}$ of $(\mu - \DL)^{-1}$ and the assumption~\eqref{int-cond2}, we estimate
\begin{align*}
\|V(\mu - \Delta_{\textnormal D,\xi})^{-1}f\|_1 &= \int_0^{\infty} |V\rho_{\xi}(\mu - \DL)^{-1} \rho_{\xi}^{-1}f|(x) \,dx \\ 
&= \int_0^{\infty} |V(x)| e^{\xi x} \int_0^{\infty} G_{\mu}(x,y) e^{-\xi y} |f(y)| \,dy \,dx \\
&\leqslant \int_0^{\infty} |f(y)| \int_0^{\infty} e^{\xi(x-y)} e^{-\sqrt{\mu}|x - y|}(x \wedge y)|V(x)| \,dx \,dy
\\ 
&\leqslant \int_0^{\infty} |f(y)| \int_0^{\infty} x |V(x)| \,dx \,dy \leqslant \alpha \|f\|_1\,.
\end{align*}
Now the assertion of (a) follows since $D$ is dense in $L_1\cap L_2(0,\infty)$.

(b) Let $D := \{L_1\cap L_2(m^2\lambda) \setcol \rho_\xi^{-1}f\in L_2(m^2\lambda)\}$, and let $f\in D$.
Then $(\mu - \Delta_{\textnormal D,\xi}^m)^{-1}f = m^{-1}\rho_{\xi}(\mu - \DL)^{-1}\rho_{\xi}^{-1}mf$,
and similarly as in the proof of~(a) we estimate
\begin{align*}
\|V(\mu - \Delta_{\textnormal D,\xi}^m)^{-1}f\|_{L_1(m^2\lambda)}
&= \int_0^{\infty} |Vm^{-1} \rho_{\xi}(\mu - \DL)^{-1}\rho_{\xi}^{-1}m f|(x) \,d(m^2\lambda)(x) \\ 
&= \int_0^{\infty} |V(x)| \frac{1}{x}e^{\xi x} \int_0^{\infty} G_{\mu}(x,y) e^{-\xi y} y |f(y)| \,dy \, d(m^2\lambda)(x) \\
&\leqslant \int_0^{\infty} |f(y)| \int_0^{\infty} \frac{x}{y}e^{\xi(x-y)} e^{-\sqrt{\mu}|x - y|}(x \wedge y) |V(x)| \,dx \,d(m^2\lambda)(y) \\
&\leqslant \int_0^{\infty} |f(y)| \int_0^{\infty} x |V(x)| \,dx \,d(m^2\lambda)(y) \leqslant \alpha \|f\|_{L_1(m^2\lambda)}\,.
\end{align*}
Now the assertion of (b) follows since $D$ is dense in $L_1\cap L_2(m\lambda^2)$.
\end{proof}

\begin{rem}\label{global-Kato}
The assertions of Proposition~\ref{Miyadera-Voigt-cond-exponential} have a kind of converse: below we will show that
\[
\int_0^{\infty} x|V(x)| \dx
= \lim_{\mu\to0+}\|V(\mu - \Delta_{\textnormal D})^{-1}\|_{1 \rightarrow 1}
= \lim_{\mu\to0+}\|V(\mu - \Delta_{\textnormal D}^m)^{-1}\|_{L_1(m^2\lambda) \rightarrow L_1(m^2\lambda)};
\]
since $\lim_{\mu\to0+}\|V(\mu - \Delta_{\textnormal D})^{-1}f\|_1 = \int_0^\infty \|VT(t)f\|_1 \,dt$ for all $f\in L_1\cap L_2(0,\infty)_+$ (cf.\ the proof of Proposition~\ref{estimates-perturbed-neu}), it then follows that \eqref{int-cond2} is equivalent to the global Kato class condition
\[
\int_0^\infty \|VT(t)f\|_1 \,dt \leqslant \alpha\|f\|_1 \qquad (f\in L_1\cap L_2(0,\infty)).
\]
In the same way one sees that \eqref{int-cond2} is true if and only if
\[
\int_0^\infty \|VT^m(t)f\|_{L_1(m^2\lambda)} \,dt \leqslant \alpha\|f\|_{L_1(m^2\lambda)} \qquad (f\in L_1\cap L_2(m^2\lambda)).
\]
Remarkably, \eqref{int-cond2} is thus equivalent to both an unweighted and a weighted global Kato class condition on~$V$.

Inequality $\lim_{\mu\to0+}\|V(\mu - \Delta_{\textnormal D})^{-1}\|_{1 \rightarrow 1} \leqslant \int_0^{\infty} x|V(x)| \dx$ holds by the proof of Proposition~\ref{Miyadera-Voigt-cond-exponential}(a).
For the converse inequality let $n\in\N$ and observe that $G_\mu(x,y) \uparrow x\wedge y$ as $\mu\downarrow0$, for all $x,y>0$.
Then by the monotone convergence theorem and a computation as in the proof of Proposition~\ref{Miyadera-Voigt-cond-exponential} one sees that
\[
\lim_{\mu\to0+}\|V(\mu - \Delta_{\textnormal D})^{-1}f\|_1 = \int_0^\infty f(y) \int_0^\infty (x\wedge y)|V(x)| \dx \dy \geqslant \|f\|_1 \int_0^\infty (x\wedge n)|V(x)| \dx
\]
for all $f\in L_1\cap L_2(0,\infty)_+$ with $\operatorname{spt} f \subseteq (n,\infty)$.
Therefore $\lim_{\mu\to0+}\|V(\mu - \Delta_{\textnormal D})^{-1}\|_{1 \to 1} \geqslant \int_0^\infty (x\wedge n)|V(x)| \dx$, and letting $n\to\infty$ one obtains the desired converse inequality.

Noting that $\frac{x}{y} G_\mu(x,y) \uparrow \frac{x^2}{y}\wedge x$ as $\mu\downarrow0$ one shows the second equality $\int_0^{\infty} x|V(x)| \dx
= \lim_{\mu\to0+}\|V(\mu - \Delta_{\textnormal D}^m)^{-1}\|_{L_1(m\lambda^2) \rightarrow L_1(m\lambda^2)}$ in a similar way, now using $f$ with $\operatorname{spt} f \subseteq (0,1/n)$.
\end{rem}

Using Proposition~\ref{Miyadera-Voigt-cond-exponential}(a) with $\xi=0$, we now show that \eqref{int-cond2} implies the form smallness condition~\eqref{form-small-dirichlet}.

\begin{prop}\label{ic-impl-fs} 
Let \eqref{int-cond2} be true. Then 
\begin{equation}\label{form-small-dirichlet2}
\int_0^{\infty} |V||u|^2 \,dx \leqslant \alpha \langle -\DL u,u \rangle \qquad (u \in \dom(\DL)).        
\end{equation}
\end{prop}
\begin{proof}
As in the proof of Proposition~\ref{Miyadera-Voigt-cond-exponential} we assume without loss of generality that $V \in L_{\infty}(0,\infty)$.

Let $p \in (1,1/\alpha)$. Then
$\|\mathopen{-p|V|(\lambda - \DL)^{-1}}\|_{1 \rightarrow 1} \leqslant p\alpha$ for all $\lambda > 0$ by Proposition~\ref{Miyadera-Voigt-cond-exponential}(a). Therefore, since $T$ is contractive and $L_1$-contractive,
\[
\|T_{-p|V|}(t)\|_{1 \rightarrow 1} \leqslant \frac{1}{1-p\alpha} \qquad (t \geqslant 0)
\]
by Corollary~\ref{estimates-perturbed}. Using duality and the self-adjointness of $T_{-p|V|}$, we also get the estimate 
\[
\|T_{-p|V|}(t)\|_{\infty \rightarrow \infty} \leqslant \frac{1}{1-p\alpha} \qquad (t \geqslant 0).
\]
The Riesz-Thorin interpolation theorem now implies that $\|T_{-p|V|}(t)\|_{2 \rightarrow 2} \leqslant 1/(1-p\alpha)$ for all $t \geqslant 0$. Thus $-(\DL + p|V|)$ is accretive, i.e.,
\[
\int_0^{\infty} p|V||u|^2 \,dx \leqslant \langle -\DL u,u  \rangle \qquad (u \in \dom(\DL)).  
\]
Letting $p \rightarrow 1/\alpha$ we obtain the asserted estimate~\eqref{form-small-dirichlet2}.
\end{proof}

Now
we are ready to prove the $L_1$-estimates mentioned at the beginning of this section. We start with the exponentially weighted $L_1$-estimate for Schrödinger semigroups on $(0,\infty)$. 
Note that $V$ is $T$-admissible by Proposition~\ref{ic-impl-fs} and Proposition~\ref{admissibility}.
\begin{prop}\label{l1-estimate-exponential}
Let \eqref{int-cond2} be true. Let $\xi \in \R$. Then
\[
\|\rho_{\xi}T_V(t)\rho_{\xi}^{-1}\|_{L_1(0,\infty) \rightarrow L_1(0,\infty)} \leqslant \frac{1}{1 - \alpha} e^{\xi^2t} \qquad (t \geqslant 0).
\]
\end{prop}
\begin{proof}
Recall the definition of $T^\xi$ from the first paragraph of the present subsection.
By~\eqref{Txi-est} we have $\|T^{\xi}(t)\|_{1 \rightarrow 1} \leqslant e^{\xi^2t}$ and $\|T^{\xi}(t)\|_{2 \rightarrow 2} \leqslant e^{\xi^2t}$ for all $t \geqslant 0$.
Now fix $n \in \N$ and let $V_n:=(V \wedge n) \vee (-n) \in L_{\infty}(0,\infty)$. Then 
\begin{equation}\label{eq-bound-pert-weig-semi}
(T^{\xi})_{V_n}(t)f = \rho_{\xi}T_{V_n}(t)\rho_{\xi}^{-1}f \qquad (t \geqslant 0,\ f \in \dom(\rho_{\xi}^{-1})) 
\end{equation}
by Proposition~\ref{weighted-scontinuity}. Moreover, we have $\|\mathopen{V_n(\lambda - \Delta_{\textnormal D,\xi})^{-1}}\|_{1 \rightarrow 1} \leqslant \alpha$ for all $\lambda > \xi^2$, by Proposition~\ref{Miyadera-Voigt-cond-exponential}(a). Hence, Corollary~\ref{estimates-perturbed} implies that
\[
\|(T^{\xi})_{V_n}(t)\|_{1 \rightarrow 1} \leqslant \frac{1}{1 - \alpha}e^{\xi^2t} \qquad (t \geqslant 0).
\]
Now the assertion follows from \eqref{eq-bound-pert-weig-semi} and the fact that $T_{V_n}(t) \rightarrow T_V(t)$ strongly for all $t \geqslant 0$.
\end{proof}

The next proposition deals with the exponentially weighted $L_1$-estimates for Schrödinger semigroups on $((0,\infty),m^2\lambda)$.

\begin{prop}\label{l1-estimate-double}
Let \eqref{int-cond2} be true. Let $\xi \in \R$. Then
\[
\|m^{-1}\rho_{\xi}T_V(t)\rho_{\xi}^{-1}m\|_{L_1(m^2\lambda) \rightarrow L_1(m^2\lambda)} \leqslant \frac{2\sqrt2}{1 - \alpha}e^{2\xi^2t} \qquad (t \geqslant 0).
\]
Moreover, if $\xi \leqslant 0$, then 
\[
\|m^{-1}\rho_{\xi}T_V(t)\rho_{\xi}^{-1}m\|_{L_1(m^2\lambda) \rightarrow L_1(m^2\lambda)} \leqslant \frac{1}{1 - \alpha}e^{\xi^2t} \qquad (t \geqslant 0).
\]
\end{prop}
\begin{proof}
Let $T^{\xi,m}$ be defined as in the paragraph preceding Proposition~\ref{Miyadera-Voigt-cond-exponential}, and recall \eqref{l2-estimate-exp-m}. Moreover, note that Proposition~\ref{stoch-ultra-double-weighted} yields
\[
\|T^{\xi,m}(t)\|_{L_1(m^2\lambda) \to L_1(m^2\lambda)} \leqslant c_\xi^{3/2}e^{c_{\xi}\xi^2t} \qquad (t \geqslant 0),
\]
where $c_{\xi}:= 2$ if $\xi > 0$ and $c_{\xi}:= 1$ if $\xi \leqslant 0$.
Now one completes the proof by the same argument as in the proof of Proposition~\ref{l1-estimate-exponential}, using part~(b) of Proposition~\ref{Miyadera-Voigt-cond-exponential} instead of part~(a).
\end{proof}

With the above weighted $L_1$-estimates at hand we can finally prove our main result.

\begin{proof}[Proof of Theorem~\ref{thm-main}]
It follows from Propositions~\ref{ic-impl-fs} and~\ref{admissibility} that $V$ satisfies the form smallness condition~\eqref{form-small-dirichlet}
and hence is $T$-admissible.
Now Propositions \ref{l1-estimate-exponential} and~\ref{l1-estimate-double} show that Corollary~\ref{kernel-bound-Schr-semi} is applicable, and this gives the desired kernel estimate.
\end{proof}

To conclude, we prove a kind of converse to Theorem~\ref{thm-main}, for negative potentials.

\begin{thm}\label{thm-necessary}
Let $V$ be $T$-admissible, and assume that $V \leqslant 0$. Further, assume that for every $t > 0$ the operator $T_V(t)$ has an integral kernel $k_t^V\in L_{\infty}((0,\infty) \times (0,\infty))$ such that \eqref{the-kernel-estimate} holds for some $c > 0$ independent of $t$. Then 
\begin{equation}\label{almost-int-cond}
\int_0^{\infty} x |V(x)| \dx < \infty.    
\end{equation}
\end{thm}
\begin{proof}
For $t > 0$ let $k_t\colon (0,\infty) \times (0,\infty) \to (0,\infty)$ be the integral kernel of $T(t)$; see \eqref{kt1}. One easily sees that 
\[
\left(1 + \frac{(x-y)^2}{4t}\right)^{3\alpha/4} \leqslant M e^{(x-y)^2/8t} \qquad (t,x,y > 0)
\]
for some $M \geqslant 1$, so it follows from \eqref{the-kernel-estimate} and \eqref{kernel-est-0} that 
\[
k_t^V(x,y) \leqslant c M \left(1 \wedge \frac{xy}{t}\right)t^{-1/2}e^{-(x-y)^2/8t} \leqslant 4\sqrt{2}\,c M k_{2t}(x,y) \qquad (\text{a.e. } x,y \in (0,\infty))
\]
for all $t > 0$. With $C:=(4\sqrt{2}\,c M) \vee 1$ we deduce that $T_V(t) \leqslant CT(2t)$ and hence $(T^m)_V(t) = (T_V)^m(t) \leqslant C T^m(2t)$ for all $t \geqslant 0$. This together with the $L_1$-stochasticity of $T^m$ (see Remark~\ref{stoch-ultra-double-weighted-rem}) implies that 
\begin{equation}\label{L1-contract-TV}
\sup_{t \geqslant 0} \|(T^m)_V(t)\|_{L_1(m^2\lambda) \to L_1(m^2\lambda)} \leqslant C.
\end{equation}

Now $T^m(t)|_{L_1 \cap L_2(m^2\lambda)}$ extends to a stochastic operator $T_1^m(t)$ on $L_1(m^2\lambda)$,
and the mapping $T_1^m \colon [0,\infty) \to \mathcal{L}(L_1(m^2\lambda))$ thus defined is a positive $C_0$-semigroup by \cite[Thm.~7]{Vo92}. Further observe that, due to \eqref{L1-contract-TV}, we can apply \cite[Prop.~3.1(a)]{Vo86} and obtain the $T_1^m$-admissibility of $V$ and the identity $(T_1^m)_V(t)f = (T^m)_V(t)f$ for all $f \in L_1 \cap L_2(m^2\lambda)$.
Therefore, Proposition~\ref{Global-Kato} implies that $V$ satisfies the global Kato class condition 
\[
\int_0^{\infty} \|VT_1^m(t)f\|_1 \,dt \leqslant (C-1) \|f\|_1 \qquad (f \in L_1(m^2\lambda)).
\]
The asserted inequality \eqref{almost-int-cond} now follows from Remark~\ref{global-Kato}.
\end{proof}

\textbf{Acknowledgement.} 

\sloppy
This paper is a contribution to the project M1 of the Collaborative Research Centre TRR 181 ``Energy Transfers in Atmosphere and Ocean" funded by the Deutsche Forschungsgemeinschaft (DFG, German Research Foundation) under project number 274762653.

\fussy


\begin{thebibliography}{ArBa93}

\bibitem[ArBa93]{ArBa93} {\sc W.\;Arendt and C.J.K.\;Batty}: Absorption semigroups and Dirichlet boundary conditions, {\it Math. Ann. \bf 295} (1993), no.\,3,  427--448.

\bibitem[Cou90]{cou90} {\sc Th.\;Coulhon}:
Dimension a l'infini d'un semi-groupe analytique,
{\it Bull. Sci. Math. \bf 114} (1990), no.\,4, 485--500.

\bibitem[DaPa89]{DaPa89} {\sc E.B.\;Davies and M.M.H.\;Pang}: Sharp heat kernel bounds for some Laplace operators, {\it Quart. J. Math. Oxford Ser. (2) \bf 40} (1989), no.\,159, 281--290.

\bibitem[Dav87]{Dav87} {\sc E.B.\;Davies}: Explicit constants for Gaussian upper bounds on heat kernels, {\it Amer. J. Math. \bf 109} (1987), no.\,2, 319--333.

\bibitem[EnNa00]{EnNa00} {\sc K.-J.\;Engel and R.\;Nagel}: {\it One-Parameter Semigroups for Linear Evolution Equations}, Graduate Texts in Mathematics, vol. {\bf 194}, Springer Verlag, New York, 2000.

\bibitem[Haa07]{haa07} {\sc M.\;Haase},
Convexity inequalities for positive operators,
{\it Positivity \bf 11} (2007), no.\,1, 57--68.

\bibitem[LiSe98]{LiSe98} {\sc V.\;Liskevich and Y.\;Semenov}: Two-sided estimates of the heat kernel of the Schrödinger operator, {\it Bull. London Math. Soc. \bf 30} (1998), no.\,6, 596--602.

\bibitem[Kat95]{Ka95} {\sc T.\;Kato}: {\it Perturbation Theory for Linear Operators}, corrected printing of 2nd edition, Springer Verlag, Berlin-Heidelberg, 1995.

\bibitem[Mol75]{mol75} {\sc S.A. Molchanov}: Diffusion processes, and Riemannian geometry,
{\it Uspehi Mat. Nauk \bf 30} (1975), no.\,1, 3--59.

\bibitem[Ouh05]{ouh05} {\sc E.M.\;Ouhabaz}:
{\it Analysis of heat equations on domains},
London Mathematical Society Monographs Series {\bf 31},
Princeton University Press, Princeton, NJ, 2005.

\bibitem[Ouh06]{Ouh06} {\sc E.M.\;Ouhabaz}: Sharp Gaussian bounds and $L^p$-growth of semigroups associated with elliptic and Schrödinger operators,
{\it Proc. Amer. Math. Soc. \bf 134} (2006), no.\,12, 3567--3575.

\bibitem[Sik04]{sik04} {\sc A.\;Sikora}: Riesz transform, Gaussian bounds and the method of wave equation,
{\it Math. Z. \bf 207} (2004), no.\,3, 643--662.

\bibitem[St\kern-.08emVo96]{stvo96} {\sc P.\;Stollmann and J.\;Voigt}:
Perturbation of Dirichlet forms by measures.
{\it Potential Anal. \bf 5} (1996), no.\,2, 109--138.

\bibitem[Voi77]{Vo77} {\sc J.\;Voigt}: On the perturbation theory for strongly continuous semigroups, {\it Math. Ann. \bf 229} (1977), no.\,2, 163--171.

\bibitem[Voi86]{Vo86} {\sc J.\;Voigt}: Absorption semigroups, their generators, and Schrödinger semigroups, {\it J. Funct. Anal \bf 67} (1986), no.\,2, 167--205.

\bibitem[Voi88]{Vo88} {\sc J.\;Voigt}: Absorption semigroups, {\it J. Operator theory \bf 20} (1988), no.\,1, 117--131.

\bibitem[Voi92]{Vo92} {\sc J.\;Voigt}: One-parameter semigroups acting simultaneously on different $L_p$-spaces, {\it Bull. Soc. Roy. Sci. Liège \bf 61} (1992), no.\,6, 465--470.

\bibitem[Vog10]{Vog10} {\sc H.\;Vogt}:
\emph{Perturbation theory for parabolic equations}.
Habilitationsschrift, TU Dresden, 2010.

\bibitem[Vog19]{vog19} {\sc H.\;Vogt}: $L_\infty$-estimates for the torsion function and $L_\infty$-growth of semigroups satisfying Gaussian bounds,
{\it Potential Anal. \bf 51} (2019), no.\,1, 37--47.

\end{thebibliography}
\end{document}